\documentclass[letterpaper,11pt,reqno]{amsart}

\makeatletter
\usepackage{amssymb}
\usepackage{latexsym}
\usepackage{amsbsy}
\usepackage{amsfonts}
\usepackage{hyperref}
\usepackage{graphicx}
\usepackage{enumerate}
\usepackage{color}
\usepackage{soul,xcolor}
\usepackage[margin=1in]{geometry}

\usepackage{tikz}

\def\marginpar#1{\ignorespaces}

\DeclareMathOperator\inv{inv}
\DeclareMathOperator\cyc{Cyc}
\DeclareMathOperator\blk{Blk}
\DeclareMathOperator\cmp{Cmp}

\newtheorem{theorem}{Theorem}[section]
\newtheorem{lemma}[theorem]{Lemma}
\newtheorem{proposition}[theorem]{Proposition}
\newtheorem{corollary}[theorem]{Corollary}
\newtheorem{definition}[theorem]{Definition}

\newtheorem{remark}[theorem]{Remark}
\numberwithin{equation}{section}
\makeatother
\begin{document}
\title[Regenerative permutations]{Regenerative random permutations of integers}

\author[Jim Pitman]{{Jim} Pitman}
\address{Statistics department, University of California, Berkeley. Email: 
} \email{pitman@stat.berkeley.edu}

\author[Wenpin Tang]{{Wenpin} Tang}
\address{Statistics department, University of California, Berkeley. Email: 
} \email{wenpintang@stat.berkeley.edu}

\date{\today} 
\begin{abstract}
Motivated by recent studies of large Mallows$(q)$ permutations, we propose a class of random permutations of $\mathbb{N}_{+}$ and of $\mathbb{Z}$, called {\em regenerative permutations}. 
Many previous results of the limiting Mallows$(q)$ permutations are recovered and extended.
Three special examples: blocked permutations, $p$-shifted permutations and $p$-biased permutations are studied.
\end{abstract}

\maketitle


\setcounter{tocdepth}{1}
\tableofcontents
\setstcolor{red}
\section{Introduction and main results}

Random permutations have been extensively studied in combinatorics and probability theory. 
They have a variety of applications including:
\begin{itemize}
\item
statistical theory, e.g. Fisher-Pitman permutation test \cite{Fisher,Pitman1}, ranked data analysis  \cite{Critchlow,DiaconisRank};
\item 
population genetics, e.g. Ewens' sampling formula \cite{Ewens72} for the distribution of allele frequencies in a population with neutral selection;
\item
quantum physics, e.g. spatial random permutations \cite{U08,BU09} arising from the Feynman representation of interacting Bose gas; 
\item
computer science, e.g. data streaming algorithms \cite{Muth,HLMV}, interleaver designs for channel coding \cite{DP,BM}.
\end{itemize}

Interesting mathematical problems are $(i)$ understanding the asymptotic behavior of large random permutations, and $(ii)$ generating a sequence of consistent random permutations. 
Over the past few decades, considerable progress has been made in these two directions:
\begin{enumerate}[$(i)$]
\item
Shepp and Lloyd \cite{SL}, Vershik and Shmidt \cite{VS1,VS2} studied the distribution of cycles in a large uniform random permutation. 
The study was extended by Diaconis, McGrath and Pitman \cite{DMP}, Lalley \cite{Lalley} for a class of large non-uniform permutations.
Hammersley \cite{Hammersley} first considered the longest increasing subsequences in a large uniform random permutation. 
The constant in the law of large numbers was proved by Logan and Shepp \cite{LS}, Kerov and Vershik \cite{KV} via representation theory, and by Aldous and Diaconis \cite{AD}, Sepp\"al\"ainen \cite{Sepp} using probabilistic arguments. 
The Tracy-Widom limit was proved by Baik, Deift and Johansson \cite{BDJ}.
See also Romik \cite{Romik}.
Recently, limit theorems for large Mallows permutations have been considered by Mueller and Starr \cite{MS}, Bhatnagar and Peled \cite{BP}, Basu and Bhatnagar \cite{BB}, Gladkich and Peled \cite{GP}.
\item
Pitman \cite{Pitman95,Pitmanbook} provided a sequential construction of random permutations of $[n]$ with consistent cycle structures. 
This is known as the {\em Chinese restaurant process}, or {\em virtual permutations} \cite{KOV93,KOV} in the Russian literature.
A description of the Chinese restaurant process in terms of records was given by Kerov \cite{Kerov}, Kerov and Tsilevich \cite{KS95}. See also Pitman \cite{Pitman17}.
Various families of consistent random permutations have been devised by Gnedin and Olshanski \cite{GO06,GO09,GO12}, Gnedin \cite{G11}, Gnedin and Gorin \cite{GG15,GG16} in a sequential way, and by Fichtner \cite{Fichtner}, Betz and Ueltschi \cite{BU09}, Biskup and Richthammer \cite{BR} in a Gibbsian way.
\end{enumerate}

The inspiration for this article is a series of recent studies of random permutations of countably infinite sets by Gnedin and Olshanski \cite{GO09,GO12}, Basu and Bhatnagar \cite{BB}, Gladkich and Peled \cite{GP}.
Here a permutation of a countably infinite set is a bijection of that set.
Typically, these models are obtained as limits in distribution, as $n \to \infty$, of some sequence of random permutations $\Pi^{[n]}$, with some given distributions $Q_n$ on the set $\mathfrak{S}_n$ of permutations of the finite set $[n]:= \{1, \ldots, n\}$.
The distribution of a limiting bijection $\Pi: \mathbb{N}_{+} \to \mathbb{N}_{+}$ is then defined by
\begin{equation}
\label{limperm}
\mathbb{P}( \Pi_i = n_i, 1 \le i \le k ) := \lim_{n \to \infty} \mathbb{P} ( \Pi^{[n]}_i = n_i, 1 \le i \le k),
\end{equation}
for every sequence of $k$ distinct values $n_i \in \mathbb{N}_{+}:=\{1,2,\ldots\}$, provided these limits exist and sum to  $1$ over all choices of $(n_i, 1 \le i \le k ) \in \mathbb{N}_{+}^k$. 
It is easy to see that for $Q_n = U_n$ the uniform distribution on $\mathfrak{S}_n$, the limits in \eqref{limperm} are identically equal to $0$, so this program fails to produce a limiting permutation of $\mathbb{N}_{+}$. 
However, it was shown by Gnedin and Olshanski \cite[Proposition A.1]{GO09} that for every $0 < q < 1$ this program is successful for $Q_n = M_{n,q}$, the {\em Mallows$(q)$ distribution on $\mathfrak{S}_n$} \cite{Mallows}, which assigns each permutation $\pi$ of $[n]$ probability
\begin{equation}
\label{mallowsn}
\mathbb{P}( \Pi^{[n]} = \pi) = M_{n,q}(\pi):= Z_{n,q}^{-1} \,\, q^{\inv(\pi)}  \qquad \mbox{for}~ \pi \in \mathfrak{S}_n,
\end{equation}
where $\inv (\pi):= \{ (i,j): 1 \le i < j \le n, \pi(i) > \pi(j) \}$ is the {\em number of inversions} of $\pi$,
and the normalization constant $Z_{n,q}$ is well known to be the {\em $q$-factorial function}
\begin{equation}
\label{qfac}
Z_{n,q} = \prod_{j=1}^n \sum_{i = 1}^j q^{i-1} = (1-q)^{-n} \prod_{j = 1 }^n ( 1 - q^j) \qquad \mbox{for}~0 < q <1.
\end{equation}
See Diaconis and Ram \cite[Section 2.e]{DR} for algebraic properties of Mallows($q$) distributions, and additional references. 
Note that it is possible to define the projective limit for both $Q_n = U_n$ and $Q_n = M_{n,q}$:
\begin{itemize}
\item
For $Q_n = U_n$, the consistency of the family $(U_n;~n \geq 1)$ with respect to the projection is closely related to the {\em Fisher-Yates-Durstenfeld-Knuth shuffle} \cite[Section 3.4.2]{Knuth2}. 
The projective limit is the Chinese restaurant process with $\theta = 1$. 
\item
For $Q_n = M_{n,q}$, the fact that $(M_{n,q};~n \geq 1)$ are consistent relative to the projection is a consequence of the {\em Lehmer code} \cite[Section 5.1.1]{Knuth3}.
Moreover, Gnedin and Olshanski \cite[Proposition A.6]{GO09} proved that the projective limit coincides with the limit in distribution \eqref{limperm}.
\end{itemize}

Gnedin and Olshanski \cite{GO09} gave a number of other characterizations of the limiting distribution of $\Pi$ so obtained for each $0 < q < 1$.
They continued in \cite{GO12} to show that there exists a two-sided random permutation $\Pi^*$ of $\mathbb{Z}$, which is a similar limit in distribution of Mallows$(q)$ permutations of $[n]$, shifted to act on intervals of integers
$[1 - a_n, n - a_n ]$, for any sequence of integers $a_n$ with both $a_n \to \infty$ and $n - a_n \to \infty$ as $n \to \infty$.
They also showed that for each $0 < q < 1$ the process $\Pi^*$ is {\em stationary}, meaning that the {\em process of displacements} $(D^{*}_z:=\Pi^*_z-z;~z \in \mathbb{Z})$ is a stationary process:
\begin{equation}
(D^*_z;~z \in \mathbb{Z}) \stackrel{(d)}{=} (D^*_{a+z};~z \in \mathbb{Z}) \qquad \mbox{for}~a \in \mathbb{Z}.  
\end{equation}
These results were further extended by Basu and Bhatnagar \cite{BB}, Gladkich and Peled \cite{GP}, who established a number of properties of the limiting Mallows$(q)$ permutations of $\mathbb{N}_{+}$ and of $\mathbb{Z}$, as well as provided many finer asymptotic results regarding the behavior of various functionals of Mallows$(q)$ permutations of $[n]$, including cycle counts, and longest increasing subsequences, in various finer limit regimes with $q$ approaching either $0$ or $1$ as $n \to \infty$.
The analysis of limiting Mallows$(q)$ permutations $\Pi$ of $\mathbb{N_{+}}$ by these authors relies on a key regenerative property of these permutations, which is generalized in this paper to provide companion results for a much larger class of random permutations of $\mathbb{N}_{+}$ and of $\mathbb{Z}$.

For a permutation $\Pi$ of a countably infinite set $I$, however it may be constructed, there is the basic question:
\begin{equation}
\parbox{30em}{
$\bullet$ {\em \mbox{ is every orbit of } $\Pi$ \mbox{ finite ?}}
}
\end{equation}
If so, say {\em $\Pi$ has only finite cycles}. 
For $I= \mathbb{N}_{+}$ or $\mathbb{Z}$, one way to show $\Pi$ has only finite cycles, and to gain some control on the distribution of cycle lengths, is to establish the stronger property:
\begin{equation}
\parbox{30em}{
$\bullet$ {\em \mbox{ every component of } $\Pi$ \mbox{ has finite length.} }
}
\end{equation}
Here we need some vocabulary. 
Let $I \subseteq \mathbb{Z}$ be an interval of integers, and $\Pi : I \rightarrow I$ be a permutation of $I$. 
Call $n \in I$ a {\em splitting time} of $\Pi$, or say that $\Pi$ {\em splits at }$n$, if $\Pi$ maps $(-\infty, n]$ red onto itself, or equivalently, $\Pi$ maps $I \cap [n+1 , \infty)$ onto itself.
The set of splitting times of $\Pi$, called the {\em connectivity set} by Stanley \cite{Stanley05}, is the collection of finite right endpoints of some finite or infinite family of {\em components of $\Pi$}, say $\{ I_j \}$.
So $\Pi$ acts on each of its components $I_j$ as an {\em indecomposable permutation} of $I_j$, meaning that $\Pi$ does not act as a permutation on any proper subinterval of $I_j$.
These components $I_j$ form a partition of $I$, which is coarser than the partition by cycles of $\Pi$. 
For example, the permutation $\pi = (1)(2,4)(3) \in \mathfrak{S}_4$ induces the partition by components $[1] [2,3,4]$.
A block of $\Pi$ is a component of $\Pi$, or a union of adjacent components of $\Pi$.
For any {\em block} $J$ of $\Pi$ with $\#J = n$ (resp. $\# J = \infty$), the {\em reduced block of $\Pi$ on $J$} is the permutation of $[n]$ (resp. $\mathbb{N}_{+}$) defined via conjugation of $\Pi$ by the shift from $J$ to $[n]$ (resp. $\mathbb{N}_{+}$).

For any permutation $\Pi$ of $\mathbb{N}_{+}$, there are two ways to express the event $\{\Pi \mbox{ splits at } n\}$  as an intersection of $n$ events:
$$
\{ \Pi \mbox{ splits at } n \} = \bigcap_{i=1}^n \{\Pi_i \le n \} =  \bigcap_{i=1}^n \{\Pi^{-1}_i \le n \}.
$$
An alternative way of writing this event is:
$$
\{ \Pi \mbox{ splits at } n \} = \bigcap_{i=1}^n  \left\{ \Pi^{-1}_i < \min_{j > n} \Pi^{-1}_j \right\} .
$$
For if $\Pi$ splits at $n$, then $\Pi^{-1}_i < n+1 = \min_{j > n} \Pi^{-1}_j$ for every $1 \le i \le n$.
Conversely, if $\min_{j > n} \Pi^{-1}_j = m+1$ say, and $\Pi^{-1}_i < m+1$ for every $1 \le i \le n$, then the image of $[n]$ via $\Pi^{-1}$ is equal to $[m]$, so $m = n$ and $\Pi^{-1}_i  \le n$ for every $1 \le i \le n$.
Let
\begin{equation}
\label{Ani}
A_{n,i}:= \left\{  \min_{j > n} \Pi^{-1}_j  < \Pi^{-1}_i \right\},
\end{equation}
be the complement of the $i^{th}$ event in the above intersection.  
Then by the principle of inclusion-exclusion 
\begin{equation}
\label{ie}
\mathbb{P}( \Pi \mbox{ splits at } n ) = 1 + \sum_{j = 1}^n (-1)^j \Sigma_{n,j},
\end{equation}
where
\begin{equation}
\label{defsigma}
\Sigma_{n,j}: = \sum_{1 \leq i_1 < \cdots <i_j \leq n} \mathbb{P}\left(\bigcap_{k = 1}^j A_{n,i_k}\right).
\end{equation}
So there are the Bonferroni bounds
\begin{equation*}
\mathbb{P}( \Pi \mbox{ splits at } n ) \ge   1  - \Sigma_{n,1}, \quad \mathbb{P}( \Pi \mbox{ splits at } n ) \le   1  - \Sigma_{n,1}  + \Sigma_{n,2},
\end{equation*}
and so on. 
Moreover, each of the intersections of the $A_{n,i}$ is an event of the form
$$
F_{B,C}:= \left\{\min_{j \in B} \Pi^{-1}_j < \min_{h \in C} \Pi^{-1}_h  \right\}, 
$$
for instance $A_{n,i} A_{n,j} = F_{B,C}$ for $F = \{n,n+1, \ldots \}$ and $C = \{i,j\}$.

An approach to the problem of whether $\Pi$ has almost surely finite component lengths for a number of interesting models, including the limiting Mallows$(q)$ model, is provided by the following structure. 
Let
$$\mathbb{N}_{+}: = \{1,2, \ldots\} \quad \mbox{and} \quad \mathbb{N}_0:= \{0,1, 2, \ldots \}.$$
If a permutation $\Pi$ of $\mathbb{N}_{+}$ splits at $n$, let $\Pi^n$ be the {\em residual permutation} of $\mathbb{N}_{+}$ defined by conjugating the action of $\Pi$ on $\mathbb{N}_{+} \setminus [n]$ by a shift back to $\mathbb{N}_{+}$: 
$$\Pi^n_i:= \Pi_{n+i} - n \qquad \mbox{for}~i \in \mathbb{N}_{+}.$$
Call $(T_n;~n \geq 0)$ a {\em delayed renewal process} if 
$$T_n := T_0 + Y_1 + \cdots + Y_n,$$
with $T_0 \in \mathbb{N}_0, Y_1, Y_2, \ldots \in \mathbb{N}_{+} \cup \{\infty\}$ independent, and the $Y_i$ identically distributed, allowing also the transient case with $\mathbb{P}(Y_1 < \infty) < 1$. 
When $T_0: = 0$, call $(T_n;~n \geq 0)$ a {\em renewal process with zero delay}.
For $n > 0$, let
$$R_n:= \sum_{k=0}^\infty 1(T_k = n),$$
be the {\em renewal indicator} at time $n$.
The definition below is tailored to the general theory of regenerative processes presented by Asmussen \cite[Chapter VI]{Asmussen}.

\begin{definition}
\label{regendef}
{\em
~
\begin{enumerate}[(1)]
\item
Call a random permutation of $\Pi$ of $\mathbb{N}_{+}$ {\em regenerative with respect to the delayed renewal process $(T_n; \, n \ge 0)$} if every $T_i$ is a splitting time of $\Pi$, and for each $n >0$ such that $\mathbb{P}(R_n=1) > 0$, conditionally given a renewal at $n$,
\begin{enumerate}[$(i).$]
\item
there is the equality in distribution 
$$
 ( \Pi^n, R_{n+1}, R_{n+2}, \ldots) \stackrel{(d)}{=} (\Pi^0, R_{1}^0, R_{2}^0, \ldots)
$$
between the joint distribution of $\Pi^n$ with the residual renewal indicators $(R_{n+1},$\\
$ R_{n+2}, \ldots)$, 
and the joint distribution of some random permutation $\Pi^0$ of $\mathbb{N}_{+}$ with renewal indicators $(R_{1}^0, R_{2}^0, \ldots)$ with zero delay;
\item
the initial segment $(R_0, R_1, \ldots R_n)$ of the delayed renewal process is independent of $( \Pi^n, R_{n+1}, R_{n+2}, \ldots)$.
\end{enumerate}
\item
Call a random permutation of $\Pi$ of $\mathbb{N}_{+}$ {\em regenerative} if $\Pi$ is regenerative with respect to some renewal process $(T_n; \, n \ge 0)$;
\item
Call a random permutation of $\Pi$ of $\mathbb{N}_{+}$ {\em strictly regenerative} if $\Pi$ is regenerative with respect to its own splitting times.
\end{enumerate}
}
\end{definition}
In Definition \ref{regendef}, we do not require the independence of the pre-renewal and the post-renewal permutations. 
This is called the {\em wide-sense regeneration} \cite[Chapter 10]{Thbook}, while assuming further independence refers to the {\em regeneration in the classical sense}.
The formulation of Definition \ref{regendef} was motivated by its application to three particular models of random permutations of $\mathbb{N}_{+}$, introduced in the next three definitions.
Each of these models is parameterized by a discrete probability distribution on $\mathbb{N}_{+}$, say $p = (p_1,p_2, \ldots)$.
These models are close in spirit to the similarly parameterized models of {\em $p$-mappings} and {\em $p$-trees}  studied in \cite{AP02, AMP}. 
See also \cite{GP05,GHP,Gnedin10} for closely related ideas of regeneration in random combinatorial structures.

General properties of a regenerative random permutation $\Pi$  of $\mathbb{N}_{+}$ with zero delay can be read from the standard theory of regenerative processes \cite[Chapter XIII]{Feller}. Let $u_0:=0$, and 
\begin{align}
u_n & := \mathbb{P}( \Pi \mbox{ regenerates at } n ), \\
f_n  & := \mathbb{P}( \Pi \mbox{ regenerates for the first time at } n ).
\end{align}
If a random permutation $\Pi$ of $\mathbb{N}_{+}$ is regenerative but not strictly regenerative, the renewal process $(T_n; \, n \ge 0)$ is not uniquely associated with $\Pi$. 
So the sequences $u_n$ and $f_n$ are not necessarily intrinsic to $\Pi$ but to $(T_n; \, n \ge 0)$.
Each of these sequences determines the other by the recursion 
\begin{equation}
\label{ufrecursion}
u_n = f_1 u_{n-1} + f_2 u_{n-2} + \cdots +f_n u_0 \quad \mbox{for all }  n > 0,
\end{equation}
which may be expressed in terms of the generating functions $U(z) : = \sum_{n=0}^{\infty} u_n z^n$ and $F(z):=\sum_{n=1}^{\infty} f_n z^n$ as 
\begin{equation}
\label{UFrelation}
U(z) = (1-F(z))^{-1}.
\end{equation}
According to the discrete renewal theorem, either
\begin{enumerate}[$(i)$]
\item
({\em transient case}) $\sum_{n=1}^\infty u_n < \infty$, when $\mathbb{P}(Y_1 < \infty ) < 1$, and $\Pi$ has only finitely many regenerations with probability one, or
\item
({\em recurrent case}) $\sum_{n=1}^\infty u_n = \infty$,  when $\mathbb{P}(Y_1 < \infty) = 1$, and with probability one $\Pi$ has infinitely many regenerations, hence only finite components, and only finite cycles.
\end{enumerate}

Here is a simple way of constructing regenerative random permutations of $\mathbb{N}_{+}$ in the classical sense:
\begin{definition}
\label{regenblock}
{\em
For a probability distribution $p:= (p_1, p_2, \ldots)$ on $\mathbb{N}_{+}$, and $Q_n$ for each $n \in \mathbb{N}_{+}$ a probability distribution on $\mathfrak{S}_n$, call a random permutation $\Pi$ of $\mathbb{N}_{+}$  {\em recurrent regenerative with block length distribution $p$ and blocks governed by $(Q_n;~n \ge 1)$}, if $\Pi$ is a concatenation of an infinite sequence {\tt Block}$_i$, $i \geq 0$ such that
\begin{enumerate}[$(i)$]
\item 
the lengths $Y_i$ of {\tt Block}$_i$, $i \geq 1$ are independent and identically distributed (i.i.d.) with common distribution $p$, and are independent of the length $T_0$ of {\tt Block}$_0$ which is finite almost surely;
\item
conditionally given the block lengths, say $Y_i= n_i$ for $i = 1, 2, \ldots$, the reduced blocks of $\Pi$ are independent random permutations of $[n_i]$ with distributions $Q_{n_i}$.
\end{enumerate}
A transient regenerative permutation $\Pi$ of $\mathbb{N}_{+}$ can be constructed as above up to time $T_N$ for $T_n$ the sum of $n$ i.i.d. variables $Y_i$ with distribution $\mathbb{P}(Y_1 = n ) = p_n/\sum_{k} p_k$, and $N$ has geometric $(1 - \sum_{k} p_k)$ distribution on $\mathbb{N}_0$, independent of the sequence $(Y_i; \, i \ge 1)$.
Then $T_N$ is the time of the last finite split point of $\Pi$, and given $T_N = n$ and the restriction of $\Pi$ on $[n]$ so created, the restriction of $\Pi$ on $[n,\infty)$, shifted back to be a permutation of $\mathbb{N}_{+}$ can be constructed according to any fixed probability distribution on the set of all permutations of $\mathbb{N}_{+}$ with no splitting times.
}
\end{definition}

The main focus here is the {\em positive recurrent case}, with mean block length $\mu:= \mathbb{E}(Y_1) < \infty$, and an aperiodic distribution of $Y_1$, which according to the discrete renewal theorem \cite[Chapter XIII, Theorem 3]{Feller} makes
\begin{equation}
\label{descren}
\lim_{n\to \infty} u_n  = 1/\mu >0 .
\end{equation}
Then numerous asymptotic properties of the recurrent regenerative permutation $\Pi$ with this distribution of block lengths can be read from standard results in renewal theory, as discussed further in Section \ref{s3}.
In particular,  starting from any positive recurrent random permutation $\Pi$ of $\mathbb{N}_{+}$,
renewal theory gives an explicit construction of a stationary, two-sided version $\Pi^*$ of $\Pi$, acting as a random permutation of $\mathbb{Z}$, along with ergodic theorems indicating the existence of limiting frequencies for various counts of cycles and components, for both the one-sided and two-sided versions.
This greatly simplifies the construction of stationary versions of the limiting Mallows$(q)$ permutations in \cite{GO12,GP}.

Observe that for every recurrent, strictly regenerative permutation of $\mathbb{N}_{+}$, the support of $Q_n$ is necessarily contained in the set $\mathfrak{S}^{\dagger}_n$ of indecomposable permutations of $[n]$.
As will be seen in Section \ref{s4}, even the uniform distribution on $\mathfrak{S}^{\dagger}_n$ has a nasty denominator for which there is no very simple formula.
The difficulty motivates the study of other constructions of random permutations of $\mathbb{N}_{+}$, such as the following:

\begin{definition}
\label{pshifted}
{\em
For $p$ a probability distribution on $\mathbb{N}_{+}$ with $p_1 >0$, call a random permutation $\Pi$ of $\mathbb{N}_{+}$ a {\em $p$-shifted permutation} of $\mathbb{N}_{+}$, if $\Pi$ has the distribution defined by the following construction from an i.i.d. sample $(X_j;~ j \ge 1)$ from $p$. Inductively,  let
\begin{itemize} 
\item 
$\Pi_1 := X_1$, 
\item 
for $i \ge 2$, let $\Pi_i := \psi(X_i)$ where $\psi$ is the increasing bijection from $\mathbb{N}_{+}$ to \\
 $\mathbb{N}_{+} \setminus \{ \Pi_1, \Pi_2, \cdots, \Pi_{i-1} \}.$
\end{itemize}
For example, if $X_1 = 2$, $X_2 = 1$, $X_3 = 2$, $X_4 = 3$, $X_5 = 4$, $X_6 = 1 \ldots$, then the associated permutation is $(2,1,4,6,8,3,\ldots)$.
}
\end{definition}

The procedure described in Definition \ref{pshifted} is a version of sampling without replacement, or {\em absorption sampling} \cite{Raw, Kemp}.
Gnedin and Olshanski \cite{GO09} introduced this construction of $p$-shifted permutations of $\mathbb{N}_{+}$ for $p$ the geometric$(1-q)$ distribution. 
They proved that the limiting Mallows$(q)$ permutations of $[n]$ is the geometric$(1-q)$-shifted permutation of $\mathbb{N}_{+}$. 
The regenerative feature of geometric$(1-q)$-shifted permutations was pointed out and exploited in \cite{BB,GP}. This regenerative feature is in fact a property of $p$-shifted permutations of $\mathbb{N}_{+}$ for any $p$ with $p_1>0$. 
This observation allows a number of previous results for limiting Mallows$(q)$ permutations to be extended as follows. 

\begin{proposition}
\label{pshiftedprop}
For each fixed probability distribution $p$ on $\mathbb{N}_{+}$ with $p_1 >0$, and $\Pi$  a $p$-shifted random permutation of $\mathbb{N}_{+}$:
\begin{enumerate}[$(i)$]
\item
The joint distribution of the random injection $(\Pi_1, \ldots, \Pi_n): [n] \to \mathbb{N}_{+}$ is given by the formula
\begin{equation}
\label{inject}
\mathbb{P}( \Pi_i = \pi_i, 1 \le i \le n) = \prod_{j=1}^n p \left( \pi_j - \sum_{1 \le i < j} 1 (\pi_i < \pi_j) \right),
\end{equation}
for every fixed injection $(\pi_i, 1 \le i \le n): [n] \to \mathbb{N}_{+}$, and $p(k):= p_k$.
\item
The probability that $\Pi$ maps $[n]$ to $[n]$ is
\begin{equation}
\label{un}
u_{n} := \mathbb{P}( [n] \mbox{ is a block of } \Pi ) = \prod_{j = 1}^n \sum_{i=1}^j p_i .
\end{equation}
\item
The random permutation $\Pi$ is strictly regenerative, with regeneration at every $n$ such that $[n]$ is a block of $\mathbb{N}_{+}$, 
and the renewal sequence $(u_n;~n \geq 1)$ as above.
\item
The distribution of component lengths $f_n:= \mathbb{P}(Y_1 = n)$ where $Y_1$ is the length of the first component of $\Pi$ is given by the probability generating function
\begin{equation}
\label{updag}
\mathbb{E} z^{Y_1} = \sum_{n=1}^\infty f_n z^n = 1 - \frac{1}{U(z)} \quad \mbox{ where } U(z) := 1 + \sum_{n=1}^ \infty u_{n} z^n  .
\end{equation}
\item
If $\mathbb{E}X_1 = m:= \sum_{i} i p_i < \infty$, then $\mu:= \mathbb{E}(Y_1) < \infty$, so $\Pi$ is positive recurrent, with limiting renewal probability
\begin{equation}
\label{limprod}
u_\infty:= \lim_{n \to \infty} u_n = \mu^{-1} = \prod_{j=1}^\infty (1 - \mathbb{P}(X_1 >j) ),
\end{equation}
Then $\Pi$ has cycle counts with limit frequencies detailed later in \eqref{limfreqs}, and there is a stationary version $\Pi^*$ of $\Pi$ acting on $\mathbb{Z}$, call it a {\em $p$-shifted random permutation of $\mathbb{Z}$.}
\item
If $ m = \infty$ then $\Pi$ is either null recurrent or transient, according to whether $U(1)$ is infinite or finite, and there is no stationary version of $\Pi$ acting on $\mathbb{Z}$.
\end{enumerate}
\end{proposition}

Even for the extensively studied limiting Mallows$(q)$ model, Proposition \ref{pshiftedprop} contains some new formulas and characterizations of the distribution, which are discussed in Section \ref{s5}.
An interesting byproduct of this proposition for a general $p$-shifted permutation is the following classical result of Kaluza \cite{Kaluza}:
\begin{corollary} \cite{Kaluza}
Every sequence $(u_n;~ n \ge 0)$ with
\begin{equation}
\label{kaluza}
0 < u_n \le u_0 = 1 \quad \mbox{and} \quad u_n^2 \le u_{n-1} u_{n+1} \mbox{ for all } n \ge 1,
\end{equation}
is a renewal sequence. The sequence $(u_n;~n \ge 0)$ satisfying \eqref{kaluza} is called a Kaluza sequence.
The renewal process associated with a Kaluza sequence is generated by the random sequence of times $n$ at which $[n]$ is a block of $\Pi$, for $\Pi$ a $p$-shifted permutation of $\mathbb{N}_{+}$, with 
$$p_1: = u_1 \quad \mbox{and} \quad p_n:= \frac{u_n}{u_{n-1}} - \frac{u_{n-1}}{u_{n-2}} \mbox{ for } n \geq 2.$$
If $p_\infty:= 1 - \sum_{i=1}^\infty p_i >0$, and $X_1, X_2, \ldots$ is the sequence of independent choices from this distribution on $\{1,2, \ldots, \infty\}$
used to drive the construction of $\Pi$, then the construction is terminated by assigning some arbitrarily distributed infinite component on $[n+1, \infty)$ following the last splitting time $n$ such that $X_1 + \cdots + X_n < \infty$, for instance by a shifting to $[n+1,\infty)$ the deterministic permutation of $\mathbb{N}_{+}$ with no finite components
$$
\cdots 6 \to 4 \to 2 \to 1 \to 3 \to 5 \to \cdots
$$
\end{corollary}
See also \cite{Kendall,Horn,KingmanR,Shanbhag,Liggett,Fristedt} for other derivations and interpretations of Kaluza's result, 
all of which now aquire some expression in terms of $p$-shifted permutations.

Some further instances of regenerative permutations are provided by the following close relative of the $p$-shifted permutation:
\begin{definition}
\label{pbiased}
{\em
For $p$ with $p_i > 0 $ for every $i$, call a random permutation $\Pi$ of $\mathbb{N}_{+}$ a {\em $p$-biased permutation of $\mathbb{N}_{+}$} if the random sequence $(p_{\Pi_1}, p_{\Pi_2},  \ldots )$ is what is commonly called a {\em sized biased random permutation} of $(p_1, p_2, \ldots)$.
That is to say,  $(\Pi_1, \Pi_2, \ldots )$ is the sequence of distinct values, in order of appearance, of a random sample of positive integers $(X_1, X_2, \ldots)$, which are independent and identically distributed (i.i.d.)
with distribution $(p_1, p_2, \ldots)$. Inductively, let
\begin{itemize} 
\item 
$\Pi_1 := X_1$,  and $J_1 := 1$,
\item 
for $i \ge 2$, let $\Pi_i := X_{J_i}$, where $J_i$ is the least $j > J_{i-1}$ such that $$X_j  \in \mathbb{N}_{+} \setminus \{X_{\Pi_1}, X_{\Pi_2}, \cdots, X_{\Pi_{i-1} } \}.$$
\end{itemize}
}
\end{definition}

The procedure described in Definition \ref{pbiased} is an instance of sampling with replacement, or {\em successive sampling \cite{Rosen, Gordon}}.
See \cite{SH,Hoppe,Donnelly,PPY,PY97} for various studies of this model of size-biased permutation, with emphasis on the annealed model, where $p$ is determined by a random discrete distribution $P:= (P_1, P_2, \ldots)$, and given $P = p$, the $X_j$ are i.i.d. with distribution $p$.  
In particular, the joint distribution of the random injection $(\Pi_1, \ldots \Pi_n): [n] \rightarrow \mathbb{N}_{+}$ is 
\begin{equation}
\label{genbias}
\mathbb{P}(\Pi_i = \pi_i, \,1 \le i \le n) = \mathbb{E} \left(P_{\pi_1} \prod_{i=2}^{n} \frac{P_{\pi_i}}{1-\sum_{j=1}^{i-1} P_{\pi_j}} \right).
\end{equation}
for every fixed injection $(\pi_i, \, 1 \leq i \leq n): [n] \rightarrow \mathbb{N}_{+}$.
A tractable model of this kind, known as a {\em residual allocation model} ({\em RAM}), has the stick-breaking representation:
\begin{equation}
\label{ram}
P_i:= (1-W_1)  \cdots (1-W_{i-1}) W_i,
\end{equation}
with $0 < W_i < 1$ and the $W_i$'s are independent and identically distributed.
This model is of special interest for Bayesian non-parametric inference and machine learning \cite{BPJ1,BPJ2}.
In those contexts, the distribution of $P$ represents a prior distribution on the underlying probability model $p$, which may be updated in response to observations such as the values in the sample $(X_1, \ldots, X_n)$, or values of $(\Pi_1, \ldots, \Pi_n)$.
A model of particular interest arises when each $W_i$ has the beta$(1,\theta)$ density $\theta (1-w)^{\theta -1 }$ at $w \in (0,1)$ for some $0 < \theta < \infty$.
This distribution of $(P_1,P_2, \ldots)$ is known as the GEM$(\theta)$ distribution, after Griffiths, Engen and McCloskey who discovered the remarkable properties of this model, including McCloskey's result that the GEM$(\theta)$ model is the only RAM that is invariant under $P$-biased permutation, meaning that there is the equality in distribution
\begin{equation}
\label{isbp}
(P_{\Pi_1}, P_{\Pi_2}, \ldots ) \stackrel{(d)}{=}  (P_1 , P_2, \ldots )  \quad \mbox{for $\Pi$ a $P$-biased permutation of $\mathbb{N}_{+}$}.
\end{equation}
The following result reveals the regeneration of sized-biased random permutations of $\mathbb{N}_{+}$.

\begin{proposition}
\label{pbiasedprop}
For every residual allocation model \eqref{ram} for a random discrete distribution $P$ with i.i.d. residual factors $W_i$, and $\Pi$ a $P$-biased random permutation of $\mathbb{N}_{+}$:
\begin{enumerate}[$(i)$]
\item
The random permutation $\Pi$ is strictly regenerative, with regeneration at every $n$ such that $[n]$ is a block of $\mathbb{N}_{+}$, and the renewal sequence $(u_n;~n \geq 1)$ defined by
\begin{equation}
\label{PTa}
u_n : = \mathbb{P}([n] \mbox{ is a block of } \Pi) = \int_0^{\infty} e^{-x} \mathbb{E} \prod_{i=1}^n \left(1 -  \exp\left(-\frac{x W_i}{T_{i}} \right) \right) dx,
\end{equation}
where $T_{i}: = (1-W_1) \cdots (1-W_{i})$.
Then $\Pi$ is positive recurrent if 
\begin{equation}
\label{gencri}
\sum_{i = 2}^{\infty} \mathbb{E}  \exp\left(-\frac{x W_i}{T_{i}} \right) < \infty \quad \mbox{for some } x > 0.
\end{equation}
\item
If each $W_i$ is the constant $1-q$ for some $0 < q < 1$, so $P$ is the geometric$(1-q)$ distribution on $\mathbb{N}_{+}$, then $\Pi$ is positive recurrent. Hence $\Pi$ has all blocks finite and limiting frequencies of cycle counts as in \eqref{limfreqs}, and there is a stationary version $\Pi^{*}$ of $\Pi$ acting on $\mathbb{Z}$, called a $p$-biased random permutation of $\mathbb{Z}$.
\item
If the $W_i$ are i.i.d. beta$(1,\theta)$ for some $\theta >0$, so $P$ has the GEM$(\theta)$ distribution, then
$\Pi$ is positive recurrent, with the same further implications.
\end{enumerate}
\end{proposition}

Propositions \ref{pshiftedprop} and \ref{pbiasedprop} expose a close affinity between $p$-shifted and $p$-biased permutations of $\mathbb{N}_{+}$, at least for some choices of $p$, which does not seem to have been previously recognized. 
For instance, if $p$ is such that $p_1$ is close to $1$, and subsequent terms decrease rapidly to $0$, then it is to be expected in either of these models that $\Pi$ should be close in some sense to the identity permutation on $\mathbb{N}_{+}$.
This intuition is confirmed by the explicit formulas described in Section \ref{s6} both for the one parameter family of geometric$(1-q)$ distributions as $q \downarrow 0$, and for the GEM$(\theta)$ family as $\theta \downarrow 0$.
This behavior is in sharp contrast to the case if $\Pi$ is a uniformly distributed permutation of $[n]$, where it is well known that the expected number of fixed points of $\Pi$ is $1$, no matter how large $n$ may be.
See also Gladkich and Peled \cite{GP} for many finer asymptotic results for the Mallows$(q)$ model of permutations of $[n]$, as both $n \to \infty$ and $q \downarrow 0$.

With further analysis, we derive explicit formulas for $u_{\infty}$ of the GEM$(\theta)$-biased permutations in Section \ref{s7}. 
But there does not seem to be any simple formula for $u_\infty$ of a $P$-biased permutation with $P$ a general RAM, and the condition \eqref{gencri} for positive recurrence is not easy to check.
Nevertheless, we give a simple sufficient condition for a $P$-biased permutation of $\mathbb{N}_{+}$ with $P$ governed by a RAM to be positive recurrent.

\begin{proposition}
\label{simplecond}
Let $\Pi$ be a $P$-biased permutation of $\mathbb{N}_{+}$ for $P$ a RAM with i.i.d. residual factors $W_i \stackrel{(d)}{=} W$. If the distribution of $- \log(1-W) < \infty$ is non-lattice, meaning that the distribution of $1-W$ is not concentrated on a geometric progression, and 
\begin{equation}
\mathbb{E} [- \log W ]< \infty \quad \mbox{and} \quad \mathbb{E} [- \log(1-W)] < \infty,
\end{equation}
 then $\Pi$ is positive recurrent regenerative permutation.
\end{proposition}

\bigskip

{\bf Organization of the paper:} The rest of this paper is organized as follows.
\begin{itemize}
\item
Section \ref{s2} sets the stage by recalling some basic properties of indecomposable permutations of a finite interval of integers, which are the basic building blocks of regenerative permutations. 
\item
Section \ref{s3} indicates how the construction of a stationary random permutation of $\mathbb{Z}$ along with some limit theorems is a straightforward application of the well established theory of regenerative random processes.
\item
Section \ref{s4} provides an example of the regenerative permutation of $\mathbb{N}_{+}$, with uniform block distribution. Some explicit formulas are given.
\item
Section \ref{s5} sketches a proof of Proposition \ref{pshiftedprop} for $p$-shifted permutations, following the template provided by \cite{BB} in the particular case of the limiting Mallows$(q)$ models.
\item
Section \ref{s6} gives a proof of Proposition \ref{pbiasedprop} for $P$-biased permutations. 
This is somewhat trickier, and the results are less explicit than in the $p$-shifted case.  
\item
Section \ref{s7} provides further analysis of regenerative $P$-biased permutations. 
There Proposition \ref{simplecond} is proved.
We also show that the limiting renewal probability of the GEM$(1)$-biased permutation is $1/3$. 
\end{itemize}

\bigskip

{\bf Acknowledgement:} 
We thank David Aldous, Persi Diaconis, Marek Biskup and Sasha Gnedin for various pointers to the literature.
Thanks to Jean-Jil Duchamps for an insightful first proof of our earlier conjecture that $u_\infty = 1/3$.
We also thank an anonymous referees for his careful reading and valuable suggestions.
\section{Indecomposable permutations}
\label{s2}

This section provides references to some basic combinatorial theory of indecomposable permutations of $[n]$ which may arise as the reduced permutations of $\Pi$ on its components of finite length.
For $1 \le k \le n$, let $(n,k)^{\dagger}$ be the number of permutations of $[n]$ with exactly $k$ components.
In particular, $(n,1)^{\dagger}:= \#  \mathfrak{S}^{\dagger}_n$ is the number of indecomposable permutations of $[n]$, as the sequence A003319 of OEIS.
As shown by Lentin \cite{Lentin} and Comtet \cite{Comtet72}, the counts $((n,1)^{\dagger};~ n \geq 1)$, starting from $(1,1)^{\dagger} = 1$, 
are determined by the recurrence
\begin{equation}
\label{indecomprec}
n! = \sum_{k = 1}^n (k,1)^{\dagger} (n-k)! ,
\end{equation}
which enumerates permutations of $[n]$ according to the size $k$ of their first component.
Introducing the formal power series which is the generating function of the sequence $(n!;~n \geq 0)$
$$
G(z):= \sum_{n=0}^\infty n! \, z^n,
$$
the recursion \eqref{indecomprec} gives the generating function of the sequence $((n,1)^{\dagger};~ n \geq 1)$, as
\begin{equation}
\label{indecompn1}
\sum_{n=1}^\infty (n,1)^{\dagger} z^n = 1 - \frac{1}{G(z)},
\end{equation}
which implies that
\begin{equation}
\label{asymf}
(n,1)^{\dagger} = n! \left(1-\frac{2}{n} + O\left(\frac{1}{n^2}\right) \right).
\end{equation}
Furthermore, it is derived from \eqref{indecompn1} that
\begin{equation}
\label{indecompnk}
\sum_{n=k}^\infty (n,k)^{\dagger} z^n = \left( 1 - \frac{1}{G(z)} \right)^k \quad \mbox{for}~1 \leq k \leq n.
\end{equation}
The identity \eqref{indecompnk} determines the triangle of numbers $(n,k)^{\dagger}$ for $1 \leq k \leq n$, as displayed for $1 \leq n \leq 10$ in Comtet \cite[Exercise VI.14]{Comtet74}. See also \cite{King,Cori1,Cori2,Cori3,AP13,BR16} for various results about indecomposable permutations.

Recall that for $I \subseteq \mathbb{Z}$ an interval of integers, and $\Pi: I \rightarrow I$ a permutation of $I$, we say $\Pi$ splits at $n \in I $, if $\Pi$ maps $I \cap (-\infty,n]$ onto itself. 
As observed by Stam \cite{Stam}, the splitting times of a uniform random permutation $\Pi$ of a finite interval of integers $I = [a,b]$ are {\em regenerative} in the sense that conditionally given that $\Pi$ splits at some $n \in I$ with $a \le n < b$, the restrictions of $\Pi$ to $[a,n]$ and to $[n+1,b]$ are independent uniform random permutations of these two subintervals of $I$.
However, for a uniform random permutation $\Pi$ of a finite interval, the components of $\Pi$ turn out not to be very interesting.
In fact, for a large finite interval of integers $I$, most permutations of $I$ have only one component. 
Assuming for simplicity that $I = [n]$,  let 
$$
V_n := \sum_{k = 1}^n 1 ( \Pi  \mbox{ splits at } k ),
$$
be the number of interval components of $\Pi$, a uniformly distributed random permutation of $[n]$.
It is easily seen from \eqref{asymf} that $\mathbb{P}(V_n = 1) \to 1$ as $n \to \infty$.
By an obvious enumeration
$$
\mathbb{E} V_n := \sum_{k = 1}^n \mathbb{P} ( \Pi  \mbox{ splits at } k ) = \sum_{k = 1}^n \frac{ k! (n-k)! } {n!}  = \Sigma_n - 1,
$$
where
$$
\Sigma_n:= \sum_{k=0}^n {n \choose k }  ^{-1},
$$
is the sum of reciprocals of binomial coefficients. 
The sum $\Sigma_n$, as the sequence A046825 of OEIS, has been studied in a number of articles \cite{Rockett,Sury}, with some other interpretations of the sum given in \cite{OEIS}.

The following lemma records some basic properties of the decomposition of a uniform permutation $\Pi$ of $[n]$.

\begin{lemma}
Let $\Pi$ be a uniformly distributed random permutation of $[n]$. Then:
\begin{enumerate}[$(i)$]
\item
The number $K_n$ of components of $\Pi$ has distribution 
\begin{equation}
\mathbb{P}(K_n = k) = \frac{(n,k)^{\dagger}}{ n!} \quad \mbox{for}~ 1 \le k \le n, 
\end{equation}
with the counts $(n,k)^{\dagger}$ determined as above.
\item
Conditionally given $K_n = k$, the random composition of $n$ defined by the lengths $L_{n,1}, \ldots, L_{n,k}$ of these components has the exchangeable
joint distribution
\begin{equation}
\mathbb{P}( L_{n,1} = n_1, \ldots, L_{n,k} = n_k \,|\, K_n = k) = \frac{1}{(n,k)^{\dagger}} \prod_{i=1}^k (n_i,1)^{\dagger},
\end{equation}
for all compositions $(n_1, \ldots, n_k)$ of $n$ with $k$ parts, meaning $n_i \ge 1$ and $\sum_{i=1}^k n_i = n$.
\item
The unconditional distribution of the length $L_{n,1}$ of the first component of $\Pi$ is given by 
\begin{equation}
\mathbb{P}(L_{n,1} = \ell ) = \frac{ (\ell,1)^{\dagger} (n- \ell)! } { n!}  \quad \mbox{for}~1 \le \ell \le n ,
\end{equation}
while the conditional distribution of $L_{n,1}$ given that $K_n = k$ is given by 
\begin{equation}
\mathbb{P}(L_{n,1} = \ell \,|\, K_n = k) = \frac{(\ell,1)^{\dagger} (n-\ell,k-1)^{\dagger} } { (n,k)^{\dagger}} \quad \mbox{for} ~1 \le \ell \le n,
\end{equation}
with the convention that $(0,0)^{\dagger} = 1$ but otherwise $(n,k)^{\dagger} = 0$ unless $1 \le k \le n$.
\item
The distribution of the length $L_{n}^{*}$ of a size-biased random component of $\Pi$, such as the length of the component of $\Pi$ containing $U_n$, where $U_n$ is independent of $\Pi$ with uniform distribution on $[n]$, is given by the formula
\begin{equation}
\label{SBCmp}
\mathbb{P}( L_n^{*} = \ell ) = \frac{ \ell \,  (\ell,1)^{\dagger}  } { n \cdot n! } \sum_{ k = 1} ^n k \, (n-l , k-1)^{\dagger},
\end{equation}
with the same convention.
\end{enumerate}
\end{lemma}

\begin{proof}
The first three parts are just probabilistic expressions of the preceding combinatorial discussion. 
Then part $(iv)$ follows from the definition of the size-biased pick, using
$$
\mathbb{P}( L_n^{*} = \ell ) = \sum_{k=1}^n \mathbb{P}( L_n^{*} = \ell \,|\, K_n = k) \mathbb{P}(K_n = k).
$$
Given that $K_n=k$ let the lengths of these $k$ components listed from left to right be $L_{n,1}, \ldots , L_{n,k}$,
\begin{align*}
\mathbb{P}(L_n^{*}  = \ell \,|\, K_n = k) &= \sum_{j=1}^k \mathbb{P}(\mbox{pick } L_{n,j} \mbox{ and } L_{n,j} = \ell \,|\, K_n = k) \\
&= k \mathbb{P}( \mbox{pick } L_{n,1} \mbox{ and } L_{n,1} = \ell \,|\, K_n = k ) \\
&= k \frac{ \ell }{n} \mathbb{P}( L_{n,1} = \ell \,|\, K_n = k),
\end{align*}
where the second equality is obtained by exchangeability. 
Now part (iv) follows by plugging in the formulas in previous parts.
\end{proof}

Table of $n\cdot n! \, \mathbb{P}(L_n^{*} = \ell)$ for $1 \le \ell \le n \le 7$:
\begin{center}
\begin{tabular}{ c | c  c  c  c  c  c  c  c }
\multicolumn{1}{l}{$n$} &&&&&&&\\\cline{1-1} 
1 &1&&&&&&& \\
2 &2&2&&&&&&\\
3 &5&4&9&&&&&\\
4 &16 &10&18&52&&&&\\
5 &64&32&45&104&355&&&\\
6 &312&128&144&260&710&2766&&\\
7 &1812&624&576&832&1775&5532&24129&\\\hline
\multicolumn{1}{l}{} &0&1&2&3&4&5&6 & ~$\ell$\\
\end{tabular}
\end{center}
\section{Regenerative and stationary permutations}
\label{s3}

This section elaborates on the structure of a regenerative permutation of $\mathbb{N}_{+}$, and its stationary version $\Pi^{*}$ acting on $\mathbb{Z}$.
To provide some intuitive language for discussion of a permutation $\Pi$  of $I= \mathbb{N}_{+}$ or of $I = \mathbb{Z}$, it is convenient to regard $\Pi$ as describing a motion of balls labeled by $I$. 
Initially, for each $i \in I$, ball $i$ lies in box $i$. After the action of $\Pi$, 
\begin{itemize}
\item ball $i$ from box $i$ is moved to box $\Pi_i$;
\item box $j$ contains the ball initially in box $\Pi^{-1}_j$.
\end{itemize}

For $i \in I$ let $D_i:= \Pi_i - i$, the displacement of ball initially in box $i$. 
It follows easily from Definition \ref{regendef} that if $\Pi$ is a regenerative permutation of $\mathbb{N}_{+}$, then the process $(D_{n};~ n \geq 1)$ is a {\em regenerative process with embedded delayed renewal process} $(T_k;~ k \geq 0)$.
This means that if $R_n:= \sum_{k=0}^\infty 1(T_k = n)$ is the $n^{th}$ renewal indicator variable, then for each 
$n$ such that $\mathbb{P}(R_n = 1) >0$,  conditionally given the event $\{R_n=1\}$,
\begin{enumerate}[$(i).$]
\item 
there is the equality of finite dimensional joint distributions
$$
 ( ( D_{n+j}, R_{n+j});~ j \geq 1 ) \stackrel{(d)}{=} ( ( D^0_{j}, R^0_{j});~ j \geq 1 ), 
$$
where the $D^0_{j}:= \Pi^0_{j} - j$ are the displacements of the random permutation $\Pi^0$ of $\mathbb{N}_{+}$, with associated renewal indicators $R^0_1, R^0_2, \ldots$ with zero delay.
\item
the bivariate process $( ( D_{n+j}, R_{n+j});~ j \geq 1 )$ is independent of $(R_1, \ldots, R_n)$.
\end{enumerate}
This paraphrases the discrete case of the general definition of a regenerative process proposed by
Asmussen \cite[Chapter VI]{Asmussen}, and leads to the following Lemma.

\begin{lemma}
\label{Asmlem}
\cite[Chapter VI, Theorem 2.1]{Asmussen}.
Let $(D_{n};~n \geq 1)$ be a regenerative process with embedded delayed renewal process, $(T_k;~k \geq 0)$,
in the sense indicated above.  
Assume that the renewal process is positive recurrent with finite mean recurrence time
$\mu:= \mathbb{E}(Y_1) < \infty$, where $Y_1:= T_1 - T_0$, and that the distribution of $Y_1$ is aperiodic.
Then there is the convergence in total variation of distributions of infinite sequences
$$
(D_n, D_{n+1}, \ldots ) \stackrel{t.v.}{\longrightarrow} (D^*_0, D^*_1, \ldots),
$$
where $(D^*_z;~ z \in \mathbb{Z})$ is a two-sided stationary process, whose law is uniquely determined by the block formula
\begin{equation}
\label{blockformula}
\mathbb{E} g( D^*_z, D^*_{z+1}, \ldots) = \frac{1}{\mu}\mathbb{E} \left( \sum_{k = 1}^\infty g(D_k, D_{k+1}, \ldots) 1 ( Y_1  \ge k) \right),
\end{equation}
for all $z \in \mathbb{Z}$ and all non-negative product measurable functions $g$.
\end{lemma}

The existence of a stationary limiting Mallows$(q)$ permutation of $\mathbb{Z}$ was established by Gnedin and Olshanski \cite{GO09}, along with various characterizations of its distribution. Their work is difficult to follow, because they did not exploit the regenerative properties of this distribution.
Gladkich and Peled \cite[Section 3]{GP} provides some further information about this model, including what they call a `stitching' construction of the two-sided model from its blocks on $(-\infty,T_0]$, $(T_0,T_1)$ and $[T_1 +1, \infty)$. But their construction too is difficult to follow. 
In fact, the structure of the two-sided Mallows permutation of $\mathbb{Z}$ is typical of the general structure of stationary regenerative processes.
This structure is spelled out in the following theorem, which follows easily from Lemma \ref{Asmlem}.
\begin{theorem}
\label{thmmain}
Let $\Pi$ be a positive recurrent regenerative random permutation of $\mathbb{N}_{+}$, with block length distribution $p$ and family of block distributions $Q_n$ on $\mathfrak{S}_n$, and $\mu: = \sum_{n} n p_n < \infty$.
\begin{enumerate}[$(i)$]
\item
There exists a unique stationary regenerative random permutation $\Pi^*$ of $\mathbb{Z}$, with associated stationary renewal process
$$\{\cdots T_{-2} < T_{-1} < T_0 < T_1 < T_2  \cdots\}  \subseteq \mathbb{Z},$$
with indexing defined by $T_{-1} < 0 \le T_0$, and renewal indicators $R_z^*$, with $R_z^* = 1$ implying that $\Pi$ splits at $z$, such that 
$$
\mathbb{P}( R_z^* = 1 ) = 1/ \mu   \quad \mbox{for}~z \in \mathbb{Z},
$$
and given the event $\{R^*_z = 1\}$, by letting $\Pi^{*,z}_{i}: = \Pi^{*}_{z+i} - z$ for $i \in \mathbb{N}_{+}$,
$$
(\Pi^{*,z}_{1}, \Pi^{*,z}_{2}, \ldots \,|\, R^{*}_z = 1) \stackrel{(d)}{=} (\Pi ^n_1, \Pi^n _2, \ldots \,|\, R_n = 1)  \quad \mbox{for } z \in \mathbb{Z},
$$
for every $n$ such that $\mathbb{P}(R_n = 1) >0$, where $(R_n;~n \geq 1)$ is the sequence of renewal indicators associated with the one-sided regenerative permutation $\Pi$.
\item
If $\Pi^{*}$ so defined, with block lengths $Y_z:= T_{z} - T_{z-1}$ for $z \in \mathbb{Z}$, then the $(Y_z;~ z \in \mathbb{Z})$ are independent, with the $Y_z, z \ne 0$ all copies of $Y_1$ with distribution $p$,
while $Y_0$ has the size-biased distribution 
$$
\mathbb{P}(Y_0 = n) = n p_n /\mu   \quad \mbox{for } n \geq 1 .
$$
Conditionally given all the block lengths, the delay $T_0$ has uniform distribution on $\{0,1, \ldots,Y_0-1\}$, and conditional on all the block lengths and on $T_0$, with given block lengths $n_i$ say, the reduced permutation of $\Pi^*$ 
on the block of $n_i$ integers $(T_{i-1}, T_{i}]$ is distributed according to $Q_{n_i}$.
\end{enumerate}
Conversely, if $\Pi$ is regenerative, existence of such a stationary regenerative permutation of $\mathbb{Z}$ implies that $\Pi$ is positive recurrent.
\end{theorem}

Also note that the law of the stationary regenerative random permutation $\Pi^{*}$ is uniquely defined by the equality of joint distributions
$$
(\Pi^{*}_1, \Pi^{*}_2, \ldots , T_0, T_1, T_2, \ldots \,|\, R^*_0 = 1) \stackrel{(d)}{=} (\Pi^0_1, \Pi^0_2, \ldots , T_0, T_1, T_2, \ldots ) ,
$$
where on the left side the $T_i$ are understood as the renewal times that are strictly positive for the stationary process $\Pi^*$, and on the right side the same notation is used for the renewal times of the regenerative random permutation $\Pi^0$ of $\mathbb{N}_{+}$ with zero delay, and on both sides $T_0 = 0$,  the $Y_i:= T_i - T_{i-1}$ for $i \ge 1$ are independent random lengths with distribution $p$, and conditionally given these block lengths are equal to $n_i$, the corresponding reduced permutations of $[n_i]$ are independent and distributed according to $Q_{n_i}$.  
So the random permutation $\Pi^0$ of $\mathbb{N}_{+}$ is a {\em Palm version} of the stationary permutation $\Pi^{*}$ of $\mathbb{Z}$. See Thorisson \cite{Th1995,Thbook} for general background on stationary stochastic processes.

Let $\Pi$ be a positive recurrent regenerative random permutation of $\mathbb{N}_{+}$, with block length distribution $p$.
For $n \in \mathbb{N}_{+}$, let
\begin{itemize}
\item
$\cyc_n$ be the length of the cycle of $\Pi$ containing $n$,
\item
$\cmp_n$ be the length of the component of $\Pi$ containing $n$,
\item
$\blk_n$ be the length of the block of $\Pi$ containing $n$.
\end{itemize}
Clearly, $1 \le \cyc_n \le \cmp_n \le \blk_n  \le \infty$, and the structure of these statistics is of obvious interest in the analysis of $\Pi$. Assuming further that $p$ is aperiodic, it follows from Lemma \ref{Asmlem} there is a limiting joint distribution of $(\cyc_n, \cmp_n, \blk_n)$ as $n \to \infty$.
However, the evaluation of this limiting joint distribution is not easy, even for the simplest regenerative models.

\quad Suppose that a large number $M$ of blocks of $\Pi$ are formed and concatenated to make a permutation of the first $N$ integers for $N \sim M \mu$ almost surely as $M \rightarrow \infty$.
Then among these $N \sim M \mu$ integers, there are about $M \ell p_{\ell}$ integers contained in regeneration blocks of length $\ell$. 
So for an integer $i = \lfloor U N \rfloor$ picked uniformly at random in $[N]$, the probability that this random integer falls in a regeneration block of length $\ell$ is approximately
$$
\mathbb{P}( \lfloor U N  \rfloor \in \mbox{regeneration block of length } \ell ) \approx \frac{M \, \ell \, p_{\ell} }{M \mu} = \frac { \ell p_{\ell} } {\mu} .
$$
This is the well known size-biased limit distribution of the length of block containing  a fixed point in a renewal process.
Now given that $\lfloor U N  \rfloor $ falls in a regeneration block of length $\ell$, the location of $\lfloor U N  \rfloor $ relative to the start of this block has uniform distribution on $[\ell]$.
These intuitive ideas are formalized and extended by the proposition below, which follows from Lemma \ref{Asmlem}, and the renewal reward theorem for ergodic averages \cite[Theorem 3.1]{Asmussen}.

\begin{proposition}
\label{cormain}
Let $\Pi$ be a positive recurrent regenerative random permutation of $\mathbb{N}_{+}$, with block length distribution $p$ with finite mean $\mu$, and blocks governed by $(Q_n;~n \geq 1)$.
\begin{enumerate}[$(i)$]
\item 
Let $C_{n,j}$ be the number of cycles of $\Pi$ of length $j$ that are wholly contained in $[n]$. Then the cycle  counts have limit frequencies
\begin{equation}
\label{limfreqs}
\lim_{n \to \infty} \frac{ C_{n,j} }{n} =  \frac{ \nu_j }{\mu}  \quad a.s. \quad \mbox{for}~j \geq 1,
\end{equation}
where $\nu_j $ is the expected number of cycles of length $j$ in a generic block of $\Pi$, and $\mu = \sum_j j \nu_j$. 
The same conclusion holds with $C_{n,j}$ replaced by the larger number of cycles of $\Pi$ of length $j$ whose least element is contained in $[n]$. 
\item
If the block length distribution $p$ is aperiodic, then
$$
\lim_{n \to \infty} \mathbb{P}(\Pi \mbox{ regenerates at } n ) = 1/\mu,   
$$
and
\begin{equation}
\lim_{n \to \infty} \mathbb{P}( \cyc_n = j ) = \frac{ j \nu_j }{ \mu } \quad \mbox{for } j \in \mathbb{N}_{+}.
\end{equation}
Alternatively, let $L_{\ell}^*$ be a random variable with values in $[\ell]$, which is the length of a sized-biased cycle of a random permutation of $[\ell]$ distributed as $Q_{\ell}$. Then
\begin{equation}
\lim_{n \to \infty} \mathbb{P}(\cyc_n = j, \blk_n = \ell) =\frac{\ell p_{\ell}}{\mu}  \mathbb{P}( L_{\ell}^{*} = j) \quad \mbox{for } 1 \leq j \leq \ell,
\end{equation}
and
\begin{equation}
\lim_{n \to \infty} \mathbb{P}(\cyc_n = j) = \frac{1}{\mu}  \sum_{\ell = 1}^\infty \ell p_{\ell} \mathbb{P}(L_{\ell}^* = j)  \quad \mbox{for~} j \geq 1.
\end{equation}
\item
Continuing to assume that $p$ is aperiodic, there is an almost sure limiting frequency $p^{\circ}_{j}$ of cycles of $\Pi$ of length $j$, relative to cycles of all lengths.
These limiting frequencies are uniquely determined by 
\begin{equation}
\label{limfreqfor1}
p^{\circ}_{j}= \frac{\nu_j}{\sum_{j=1}^{\infty} \nu_j} \quad \mbox{for } j \in \mathbb{N}_{+},
\end{equation}
or by the relations
\begin{equation}
\label{limfreqfor2}
p^{\circ}_{j}= \frac{\mu^\circ}{\mu} \frac{1}{j} \sum_{\ell = 1}^\infty \ell p_{\ell} \mathbb{P}(L_{\ell}^* = j) \quad \mbox{for } j \in \mathbb{N}_{+},
\end{equation}
with $\mu^\circ := \sum_{j = 1}^{\infty} j p^{\circ}_{j}$.
\item
The statements $(i) - (iii)$ hold with cycles replaced by components, with almost sure limiting frequencies $p^{\dagger}_j$ of components of $\Pi$ of length $j$.
\end{enumerate}
\end{proposition}
\section{Uniform blocked permutations}
\label{s4}

In this section we study an example of regenerative permutations where it is possible to describe the limiting cycle count frequencies explicitly. 
The story arises from the following observation of Shepp and Lloyd \cite{SL}.
\begin{lemma}
\cite{SL}
Let $N$ be a random variable with the geometric$(1-q)$ distribution on $\mathbb{N}_0$. That is,
$$\mathbb{P}(N = n) = q^n(1 - q) \quad \mbox{for } n \geq 0.$$ 
Let $\Pi$ be a uniform random permutation of $[n]$ given $N = n$.
Let $(N_j;~j \geq 1)$ be the cycle counts of $\Pi$, which given $N = 0$ are identically $0$, and given $N=n$ are distributed as the counts of cycles of various lengths $j$ in a uniform random permutation of $[n]$.
Then $(N_j;~j \geq 1)$ are independent Possion random variables with means
$$\mathbb{E}N_j = \frac{q^j}{j} \quad \mbox{for~} j \ge 1.$$
\end{lemma}

 The {\em L\'evy-It\^o representation} of $N$ with the infinitely divisible geometric$(1-q)$ distribution as a weighted linear combination of independent Poisson variables, is realized as $N = \sum_{j=1}^\infty j N_j$.
The possibility that $N = 0$ is annoying for concatenation of independent blocks. 
But this is avoided by simply conditioning a sequence of independent replicas of this construction on $N > 0$ for each replica.
The obvious identity $N_j 1 (N >0 ) = N_j$ allows easy computation of
\begin{equation}
\label{countmean}
\mathbb{E}( N_j \,|\, N > 0 ) = \frac{ \mathbb{E} N_j  }{ \mathbb{P} ( N > 0 ) } = \frac{ q^{j-1} }{j} \quad \mbox{for } j \ge 1.
\end{equation}
Similarly, for $k = 1,2 \ldots$
\begin{equation}
\mathbb{P}( N_j = k \, | \, N > 0 ) = \frac{1}{k! \, q} \left(\frac{q^j}{j}\right)^k \exp\left( - \frac{q^j}{j} \right) \quad \mbox{for } j \ge 1,
\end{equation}
hence by summation 
\begin{equation}
\mathbb{P}( N_j > 0 \,|\, N > 0 ) = \frac{1}{q} \left[ 1 - \exp\left( - \frac{q^j}{j} \right) \right]  \quad \mbox{for } j \ge 1.
\end{equation}

\begin{proposition}
\label{unifgeo}
Let $\Pi$ be the regenerative random permutation of $\mathbb{N}_{+}$, which is the concatenation of independent blocks of uniform random permutations of lengths $Y_1,Y_2,\ldots$ where each $Y_i>0$ has the geometric$(1-q)$ distribution on $\mathbb{N}_{+}$. Then:
\begin{enumerate}[$(i)$]
\item
The limiting cycle count frequencies $\nu_j/\mu$ in \eqref{limfreqs} are determined by the formula $\mu:= \mathbb{E}(Y_1) = (1- q)^{-1}$, and
\begin{equation}
\nu_j = \frac{q^{j-1} }{j} \quad \mbox{for } j \in \mathbb{N}_{+}.
\end{equation}
\item
The distribution of $\Pi_1$ is given by
\begin{equation}
\label{pi1}
\mathbb{P}( \Pi_1 = k )  = \frac{1-q}{q} \left( \lambda_1(q) - \sum_{h=1}^{k-1} \frac{ q^h } {h} \right) \quad \mbox{for } k \in \mathbb{N}_{+},
\end{equation}
where 
\begin{equation*}
\lambda_1(q):= \sum_{h=1}^\infty \frac{q^h}{h } = - \log ( 1 - q ).
\end{equation*}
\item
The probability of the event $\{\Pi_1 = 1, \Pi_2 = 2\}$ that both $1$ and $2$ are fixed points of $\Pi$, is 
\begin{equation}
\label{fix2}
\mathbb{P}(\Pi_1 = 1, \Pi_2 = 2) = 1-q .
\end{equation}
\item
The regenerative random permutation $\Pi$ is not strictly regenerative.
\end{enumerate}
\end{proposition}
\begin{proof}
\noindent $(i)$ This follows readily from the formula \eqref{countmean} for the cycle counts in a generic block.

\noindent $(ii)$ By conditioning on the first block length $Y_1$, since given $Y_1 = y$ the distribution of $\Pi$ is uniform on $[y]$, there is the simple computation for $k = 1,2, \ldots$
\begin{equation*}
\mathbb{P}( \Pi_1 = k ) = \sum_{y = k}^\infty q ^{y-1} ( 1 - q) \frac{1}{y},
\end{equation*}
which leads to \eqref{pi1}. In particular, the probability that $1$ is  a fixed point of $\Pi$ is
$$
\mathbb{P}( \Pi_1 = 1 ) = - \frac{1-q}{q} \log(1-q).
$$ 

\noindent $(iii)$ The joint probability of the event $\{\Pi_1 = 1, \Pi_2 = 2\}$ is computed as
\begin{align*}
\mathbb{P}( \Pi_1 = 1, \Pi_2 = 2 ) &= \mathbb{P}(Y_1 = 1, \Pi_1 = 1, \Pi_2 = 2 ) +  \mathbb{P}( Y_1 \ge 2 , \Pi_1 = 1, \Pi_2 = 2 ) \\
&= (1-q) \cdot \frac{1-q }{q} \lambda_1(q) + \sum_{y=2}^\infty q^{y-1} ( 1-q) \frac{ 1 } { y ( y - 1) }\\
&= \frac{(1-q)^2 }{q} \lambda_1(q) + \frac{ 1 - q }{q} \lambda_2(q), 
\end{align*}
where 
\begin{equation}
\lambda_2(q) := \sum_{h=2}^\infty \frac{ q^{h} } { h ( h - 1) }  = q - ( 1-q) \lambda_1(q).
\end{equation}
But this simplifies, by cancellation of the two terms involving $\lambda_1(q)$, to the formula \eqref{fix2}.

\noindent $(iv)$ This follows from the fact that $\mathbb{P}(\Pi_1 = 1, \Pi_2 = 2) \ne \mathbb{P}(\Pi_1 = 1)^2$.
\end{proof}

More generally, the probability of the event $\{\Pi_i = i, ~1 \le i \le k\}$ involves
$$
\lambda_k(q):= \sum_{h=k}^\infty \frac{q^h}{ h ( h - 1) \cdots (h-k+1)} = \frac{1}{a_k}  q p_{k-2}(q) + \frac{(q-1)^{k-1}}{(k-1)!} \lambda_1(q),$$
for some $a_k \in \mathbb{Z}$ and $p_{k-2} \in \mathbb{Z}_{k-2}[q]$. The sequence $(a_k;~ k \geq 1)$ appears to be the sequence A180170 of OEIS, see \cite{MMR} for related discussion.

\begin{remark}
{\em
One referee proposed the following regenerative but not strictly regenerative permutation. Consider a stationary $M/M/c$ service system with $c>1$ servers, a single queue and the first-come-first serve policy. Labeling customers in the arrival order, the output order is a random permutation $\Pi$. When the system turns idle, we have a renewal for $\Pi$. But there is no renewal, though a split, if the first served customers are $1,2, \ldots, n$ in some output order and the $(n+1)^{th}$ customer is still in the system.
}
\end{remark}

Proposition \ref{unifgeo} is generalized by the following one, which is a corollary of Theorem \ref{thmmain} and Proposition \ref{cormain}.

\begin{proposition}
\label{cormain2}
Let $\Pi$ be a positive recurrent random permutation of $\mathbb{N}_{+}$, whose block lengths $Y_k$ are i.i.d. with distribution $p$, and whose reduced block permutations given their lengths are uniform on $\mathfrak{S}_n$ for each length $n$. 
\begin{enumerate}[$(i)$]
\item
The limiting cycle count frequencies $\nu_j/\mu$ in \eqref{limfreqs} are determined by the formula
\begin{equation}
\label{limfrequnif}
\nu_j = j^{-1} \mathbb{P}(Y_1 \ge j)  \quad \mbox{for}~j \in \mathbb{N}_{+},
\end{equation}
where $\mathbb{P}(Y_1 \ge j)  =  \sum_{i=j}^{\infty} p_i$. So the almost sure limiting frequencies $p^{\circ}_{j}$ of cycles of $\Pi$ of length $j$ are given by
\begin{equation}
\label{limfreqfor3}
p^\circ_j = \frac{ \sum_{i = j}^{\infty} p_i}{j  \sum_{i=1}^{\infty} p_i H_i} \quad \mbox{for } j  \in \mathbb{N}_{+},
\end{equation}
where $H_i : = \sum_{j = 1}^i 1/j$ is the $i^{th}$ harmonic sum.
\item
If $p$ is aperiodic, the limit distribution of displacements $D_n: = \Pi_n - n$ as $n \to \infty$ is the common distribution of the displacement $D^*_z:=\Pi^*_z - z$ for every $z \in \mathbb{Z}$, which is symmetric about $0$, according to the formula
\begin{equation}
\label{limdispl}
\lim_{n \to \infty} \mathbb{P}( D_n= d ) = \mathbb{P}( D^*_z  = d ) = \frac{1}{\mu} \mathbb{E} \left( \frac{ (Y_1 - |d| )_{+} }{ Y_1 } \right) \quad \mbox{for~}d \in \mathbb{Z},
\end{equation}
which implies
\begin{equation}
\lim_{n \to \infty} \mathbb{P}( \Pi_n > n ) = \mathbb{P}( D^*_z > 0 ) = \frac{1}{2} \left( 1 - \frac{1}{\mu} \right),
\end{equation}
and the same holds for $<$ instead of $>$.
\item 
Continuing to assume that $p$ is aperiodic, there is also the convergence of absolute moments of all orders $r>0$
\begin{equation}
\lim_{n \to \infty} \mathbb{E} | D_n  |^r  = \mathbb{E} |D^*_z  |^r  = \frac{2}{\mu} \mathbb{E} \delta_r(Y),
\end{equation}
where
$$
\delta_r(n):= \sigma_r(n) - n^{-1} \sigma_{r+1}(n) \quad \mbox{with }  \sigma_r(n):= \sum_{k=1}^n k^r,
$$
the sum of $r^{th}$ powers of the first $n$ positive integers. In particular, for $r  \geq 1$, 
$\delta_r(n)$ is a polynomial in $n$ of degree $r+1$, for instance
$$
\delta_1(n) = \frac{1}{6} ( n^2 - 1), \quad \delta_2(n) = \frac{1}{12} n ( n^2 - 1),
$$
implying that the limit distribution of displacements has a finite absolute moment of order $r$ if and only if $\mathbb{E} Y_1^{r+1} < \infty$.
\end{enumerate}
\end{proposition}

\begin{proof}
\noindent $(i).$
Recall the well known fact that for a uniform random permutation of $[n]$, for $1 \le j \le n$ the expected number of cycles of length $j$ is $\mathbb{E} C_{n,j} = 1/j$.
This follows from the easier fact that the length of a size-biased pick from the cycles of a uniform permutation of $[n]$ is uniformly distributed on $[n]$, and the probability $1/n$ that the size-biased pick has length $j$
can be computed by conditioning on the cycle counts as $1/n = \mathbb{E}[ j C_{n,j}/n]$. 
Appealing to the uniform distribution of blocks given their lengths, given $Y_1$ the expected number of $j$-cycles in the block of length $Y_1$ is $(1/j) 1(Y_1 \ge j )$, and the conclusion follows. The limiting frequencies \eqref{limfreqfor3} are computed by injecting the formula \eqref{limfrequnif} for cycle counts into \eqref{limfreqfor1}.

\noindent $(ii).$
This follows from Lemma \ref{Asmlem}, with the expression for the limit distribution of $D^*_z:= \Pi^*_z - z$ given by
\begin{align}
\label{dzexp}
\mathbb{P} (D^*_z = d) = \frac{1}{\mu} \sum_{k = 1}^\infty \mathbb{P}( \Pi_k = k + d, Y_1 \ge k) .
\end{align}
By construction of $\Pi$, given $Y_1 = y$ for some $y \ge k$, the image of $\Pi_k$ is a uniform random pick from $[y]$, so
$$
\mathbb{P}( \Pi_k = k + d, Y_1 \ge k ,  Y_1 = y) = 1( 1 \le k + d \le y ) y^{-1} p_y \quad \mbox{for~} y \ge k.
$$
Sum this expression over $y$, then switch the order of summations over $k$ and $y$, to see that for each fixed $y \ge 1$ the coefficient of $\mu^{-1} y^{-1} p_y$ in \eqref{dzexp} is
$$
\sum_{k=1}^\infty 1 ( 1 \le k + d \le y ) = ( y - |d|)_+ ,
$$
since if $d \ge 0$ the sum over $k$ is effectively from $1$ to $y-d$, and while if $d < 0$ it is from $1 + |d|$ to $y$, and in either case the number of non-zero terms is $y - |d|$ if $|d| <  y$, and $0$ otherwise.
This gives the expression for the limit on the right side of \eqref{limdispl}, from which follow the remaining assertions.

\noindent $(iii).$
This follows from the formula \eqref{limdispl}, a known result of convergence of moments in the limit theorem for regenerative stochastic processes \cite[Chapter VI, Problem 1.4]{Asmussen}, and Bernoulli's formula for $\sigma_r(n)$ as a polynomial in $n$ of degree $r+1$, see e.g. Beardon \cite{Beardon}. 
\end{proof}

Note, however, that the companion results for components of $\Pi$ seem to be complicated. For instance, there is in general no simple expression for the expected number of components of $\Pi$ of a fixed length. The limiting frequencies $p^{\dagger}_j$ of components of $\Pi$ of length $j$ are obtained by plugging \eqref{SBCmp} into \eqref{limfreqfor2}, which are determined implicitly by the relations
\begin{equation}
p^{\dagger}_j = \frac{\mu^{\dagger}}{\mu} (j,1)^{\dagger}  \sum_{\ell = 1}^\infty  \frac{p_{\ell}}{\ell!} \sum_{k = 1}^{\ell} k \, (l-j,k-1)^{\dagger}
\quad \mbox{for } j \in \mathbb{N}_{+}.
\end{equation}
\section{$p$-shifted permutations}
\label{s5}

 In this section we study the $p$-shifted permutations introduced in Definition \ref{pshifted}. 
It is essential that $p$ be fixed and not random to make $p$-shifted permutations regenerative.
The point is that if $p$ is replaced by a random $P$, the observation of $\Pi_1, \ldots, \Pi_n$ given a split at $n$ allows some inference to be made about the $P_i$, $1 \le i \le n$. 
But according to the definition of the $P$-shifted permutation, these same values of $P_i$ are used to create the remaining permutation of $\mathbb{N}_{+} \setminus [n]$. 
Consequently, the independence condition required for regeneration at $n$ will fail for any non-degenerate random $P$.
Now we give a proof of Proposition \ref{pshiftedprop}.

\begin{proof}[Proof of Proposition \ref{pshiftedprop}]
\noindent $(i)$ This is clear from the definition of $p$-shifted permutations.

\noindent $(ii)$ This is obvious from the absorption sampling: one element of $[n]$ is sampled with probability $p_1 + \cdots + p_n$, then one element of the remaining in $[n]$ is sampled with probability $p_1 + \cdots + p_{n-1}$, and so on.
Alternatively, observe that
$$
u_n  = \sum_{\pi \in \mathfrak{S}_n}  \prod_{j=1}^n p \left(\pi_j - \sum_{1 \leq i <j} 1(\pi_i < \pi_j)\right) 
        =  \sum_{\pi \in \mathfrak{S}_n}  \prod_{j=1}^n p \left(j - \sum_{1 \leq i < \pi^{-1}_j} 1(\pi_i <j)\right),
$$
and the conclusion follows from the well known bijection $\mathfrak{S}_n \rightarrow [1] \times [2] \ldots \times [n]$ defined by
$$ \pi \mapsto  \left (j - \sum_{1 \leq i < \pi^{-1}_j} 1(\pi_i <j);~1 \le j \le n \right). $$

\noindent $(iii)$-$(iv)$ The strict regeneration is clear from the definition of $p$-shifted permutations, and the generating function \eqref{updag} follows easily from the the general theory of regenerative processes \cite[Chapter XIII]{Feller}.
 
\noindent $(v)$-$(vi)$
The particular case of these results for $p$ the geometric$(1-q)$ distribution was given by Basu and Bhatnagar \cite[Lemmas 4.1 and 4.2]{BB}. 
Their argument generalizes as follows.
The key observation is that for $X_1, X_2, \ldots$ the i.i.d. sample from $p$ which drives the construction of the $p$-shifted permutation $\Pi$, the sequence $M_n$ defined by $M_0:= 0$ and 
$$
M_{n}:= \max ( M_{n-1}, X_n ) - 1,
$$
has the interpretation that 
$$
M_n = \# \left\{ i : 1 \le i \le \max_{1 \le j \le n}  \Pi_j  \right\} - n,
$$
which can be understood as the current number of gaps in the range of $\Pi_j, 1 \le j \le n$.
The event $\{$$\Pi$ regenerates at $n$$\}$ is then identical to the event $\{M_n = 0\}$.
It is easily checked that $(M_n;~ n \geq 0)$ is a Markov chain with state space $\mathbb{N}_0$, and the unique invariant measure $(\mu_i;~ i \in \mathbb{N}_0)$ for the Markov chain $(M_n;~n \geq 0)$ is given by
\begin{equation}
\mu_0 =1 \quad \mbox{and} \quad \mu_i = \frac{\mathbb{P}(X_1>i)}{\prod_{j=1}^{i}[1 - \mathbb{P}(X_1 > j)]} \quad \mbox{for } i \geq 1.
\end{equation}
Moreover, it follows by standard analysis that this sequence $\mu_j$  is summable if and only if the mean $m$ of $X_1$ is finite.
The conclusion follows from the well known theory of Markov chains \cite[Chapter 6]{Durrett}, and Theorem \ref{thmmain}.
\end{proof}

See also Alappattu and Pitman \cite[Section 3]{AP08} for a similar argument used to derive the stationary distribution of the lengths of the loop-erasure in a loop-erased random walk.
For the $p$-shifted permutation, the first splitting probabilities $f_n:= \mathbb{P}( \Pi \mbox{ first splits at } n )$ are given by the explicit formulas
\begin{align*}
f_1 =& p_1,  \\
f_2 =& p_1 p_2 ,\\
f_3 =& p_1 p_2 ^2 + p_1^2 p_3 + p_1 p_2 p_3 , \\
f_4 =&p_1 p_2^3 + 2 p_1^2 p_2 p_3 + 2 p_1 p_2^2 p_3 + p_1^2 p_3^2 + p_1 p_2 p_3^2   \\
     & + p_1^3 p_4 + 2p_1^2 p_2p_4 + p_1 p_2^2 p_4 + p_1^2 p_3 p_4 + p_1 p_2 p_3 p_4.
\end{align*}
It is easily seen that for each $n$, $f_n(p_1,p_2, \ldots)$ is a polynomial of degree $n$ in variables $p_1, \ldots , p_n$. 
The polynomial so defined makes sense even for variables $p_i$ not subject to the constraints of a probability distribution. 
The polynomial can be understood as an enumerator polynomial for the vector of counts
$$
R_{n,j} := \pi_j - \sum_{i = 1}^j 1 (\pi_i < \pi_j)  \quad \mbox{for }1 \leq j \leq n.
$$
In the polynomial for $f_n$, the choice of $\pi_1, \ldots, \pi_n$ is restricted to the set $\mathfrak{S}^\dagger_n$ of indecomposable permutations of $[n]$, and the coefficient of $p_1^{r_1} \ldots p_n^{r_n}$ is for each choice of
non-negative integers $r_1, \ldots, r_n$ with $\sum_{i=1}^n r_i = n$ is the number of indecomposable permutations of $[n]$ such that $\sum_{j=1}^n 1 (R_{n,j} = i ) = r_i$ for each $1 \le i \le n$.
In particular, the sum of all the integer coefficients of these monomials is
$$
f_n(1,1, \ldots) =(n,1)^{\dagger},
$$
which is the number of indecomposable permutations of $[n]$ discussed in Section \ref{s2}.

Properties of the limiting Mallows$(q)$ permutations of $\mathbb{N}_{+}$ and of $\mathbb{Z}$ are obtained by specializing Proposition \ref{pshiftedprop} with $p$ the geometric$(1-q)$ distribution on $\mathbb{N}_{+}$.
Many results of  \cite{GO09,GP} acquire simpler proofs by this approach.
The following corollary also exposes a number of properties of the limiting Mallows$(q)$ models which were not 
mentioned in previous works.  

\begin{corollary}
For each $0<q<1$, with $\mathbb{P}_q$ governing $\Pi$ as a geometric$(1-q)$-shifted permutation of $\mathbb{N}_{+}$, the conclusions of Proposition \ref{pshiftedprop} apply with the following reductions:
\begin{enumerate}[$(i)$]
\item
The formula \eqref{inject} reduces to
\begin{equation}
\label{injectq}
\mathbb{P}_q( \Pi_i = \pi_i, 1 \le i \le n) = (1-q)^n \, q^{\inv(\pi) + \delta(n,\pi)} 
\end{equation}
where $\inv(\pi)$ is the number of inversions of $\pi$, and $\delta(n,\pi):= \sum_{i=1}^n \pi_i - \frac{1}{2} n ( n +1)$.
In particular, \eqref{injectq} holds with the further simplification $\delta(n,p)=0$ if and only if $\pi$ is a permutation of $[n]$. 
\item
The probability that $\Pi$ maps $[n]$ to $[n]$ is 
\begin{equation}
\label{unq}
u_{n,q} := \mathbb{P}_q( [n] \mbox{ is a block of } \Pi ) = (1-q)^n Z_{n,q},
\end{equation}
where $Z_{n,q}$ is defined by \eqref{qfac}.
\item
The $\mathbb{P}_q$ distribution of $\Pi$ is strictly regenerative, with regeneration at every $n$ such that $[n]$ is a block of $\mathbb{N}_{+}$, and renewal sequence $(u_{n,q};~n \geq 1)$ as above.
\item
The $\mathbb{P}_q$ distribution of component lengths $f_{n,q} = \mathbb{P}_q(Y_1= n)$, where $Y_1$ is the length of the first component of $\Pi$, is given by the probability generating function
\begin{equation}
\label{pnmallows}
\sum_{n = 1}^\infty f_{n,q} z^n  = 1 - \frac{1}{U_q(z)} \quad \mbox{where }  U_q(z) = 1 + \sum_{n=1}^\infty u_{q,n} z^n,
\end{equation}
as well as by the formula
\begin{equation}
\label{pnmallows}
f_{n,q} = (1- q)^n \, Z_{n,q}^\dagger,
\end{equation}
where
$
Z_{n,q}^\dagger := \sum_{\pi \in \mathfrak{S}^{+}_n }  q^{\inv(\pi)} 
$
is the restricted partition function of the Mallows$(q)$ distribution $M_{n,q}$ on the set $\mathfrak{S}_n^{\dagger} $ of indecomposable permutations of $[n]$.
\item
Under $\mathbb{P}_q$, conditionally given the component lengths, say $Y_i = n_i$ for $i = 1,2,\ldots$, the reduced components of $\Pi$ are independent random permutations of $[n_i]$ with conditional Mallows$(q)$ distributions $M_{n_i,q}^{\dagger}$ 
defined by
\begin{equation}
M_{n_i,q}^{\dagger}(\pi):= \frac{1}{Z_{n_i,q}^\dagger}\,   q^{\inv(\pi)} \quad \mbox{for } \pi \in \mathfrak{S}^{+}_{n_i}.
\end{equation}
\end{enumerate}
\end{corollary}
\section{$p$-biased permutations}
\label{s6}

This section provides a detailed study of $p$-biased permutations introduced in Definition \ref{pbiased}.
For a $P$-biased permutation $\Pi$ of $\mathbb{N}_{+}$ with $P = (P_1, P_2, \ldots)$ a random discrete distribution, the joint distribution of $(\Pi_1, \ldots \Pi_n)$ is computed by the formula \eqref{genbias}.
In particular, the distribution of $\Pi_1$ is given by the vector of means $(\mathbb{E}(P_1), \mathbb{E}(P_2), \ldots )$.
So if $P$ is the GEM$(\theta)$ distribution, then $\Pi_1$ has the geometric$(\theta/(1+\theta))$ distribution on $\mathbb{N}_{+}$.
The index $\Pi_1$ of a single size-biased pick from $(P_1,P_2, \ldots)$, and especially the random size
$P_{\Pi_1}$ of this pick from $(P_1,P_2, \ldots)$ plays an important role in the theory of random discrete distributions and associated random partitions of positive integers \cite{Pitmanbook}. 
Features of the joint distribution of $(P_{\Pi_1}, \ldots, P_{\Pi_n})$ also play an important role in this setting \cite{Pitman95}, but we are unaware of any previous study of $(\Pi_1, \Pi_2, \ldots )$
regarded as a random permutation of $\mathbb{N}_{+}$.

We start with the following construction of size-biased permutations from Perman, Pitman and Yor \cite[Lemma 4.4]{PPY}. 
See also Gordon \cite{Gordon} where this construction is indicated in the abstract, and Pitman and Tran \cite{Ptran} for further references to size-biased permutations.
\begin{lemma}
\cite{Gordon, PPY}
\label{lemPPY}
Let $(L_i;~ 1 \leq i \leq n)$ be a possibly random sequence such that $\sum_{i=1}^n L_i = 1$, and $(\varepsilon_i;~ 1 \leq i \leq n)$ be i.i.d. standard exponential variables, independent of the $L_i$'s. 
Define
$$Y_i : = \frac{\varepsilon_i}{L_i} \quad \mbox{for } 1 \leq i \leq n.$$ 
Let $Y_{(1)} < \cdots <Y_{(n)}$ be the order statistics of the $Y_i$'s, and $L^*_1, \cdots ,L^*_n$ be the corresponding $L$ values. 
Then $(L^*_i;~1 \leq i \leq n)$ is a size-biased permutation of $(L_i;~ 1 \leq i \leq n)$.
\end{lemma}

By applying the formula \eqref{ie} and Lemma \ref{lemPPY}, we evaluate the splitting probabilities for $P$-biased permutation with $P$ a random discrete distribution.
\begin{proposition}
\label{bon}
Let $\Pi$ be a $P$-biased permutation of $\mathbb{N}_{+}$ of a random discrete distribution $P = (P_1, P_2, \ldots)$, and $T_n: = 1 -\sum_{i=1}^n P_i$. Then the probability that $\Pi$ maps $[n]$ to $[n]$ is given by \eqref{ie} with 
\begin{equation}
\Sigma_{n,j}= \sum_{1 \le i_1 < \cdots < i_j \le n } \mathbb{E} \left( \frac{T_n} {T_n + P_{i_1} + \cdots + P_{i_j} } \right).
\end{equation}
\end{proposition}
\begin{proof}
Recall the definition of $A_{n,i}$ from \eqref{Ani}. By Lemma \ref{lemPPY}, for $1\leq  i_1< \cdots < i_j \leq n$,
\begin{align*}
\mathbb{P}\left( \bigcap_{k =1}^j A_{n,i_k}  \right) & =\mathbb{P} \left( \min_{1 \leq k \leq j} \left(\frac{\varepsilon_{i_k}}{P_{i_k}}\right) > \frac{\varepsilon}{T_n} \right) \\
& = \mathbb{E} \exp\left(-\frac{\varepsilon( P_{i_1} + \cdots + P_{i_j})}{T_n}\right) \\
& = \mathbb{E} \left( \frac{T_n} {T_n + P_{i_1} + \cdots + P_{i_j} } \right),
\end{align*}
which leads to the desired result.
\end{proof}

 In terms of the occupancy scheme by throwing balls independently into an infinite array of boxes indexed by $\mathbb{N}_{+}$ with random frequencies $P = (P_1, P_2,\ldots)$, the quantity $\Sigma_{n,j}$ has the following interpretation. Let $C_n$ be the count of empty boxes when the first box in $ \{n+1, n+2, \cdots\}$ is filled. Then
\begin{equation}
\Sigma_{n,j} = \mathbb{E} \binom{C_n}{j}.
\end{equation}
Further analysis of $C_n$ and $\Sigma_{n,j}$ for the GEM$(\theta)$ model will be presented in the forthcoming article \cite{DPT}.
Contrary to $p$-shifted permutations, we consider $P$-biased permutations where $P$ is determined by a RAM \eqref{ram}.
In the latter case, the only model with $P$ fixed is the geometric$(1-q)$-biased permutation for $0 < q < 1$.
Now we give a proof of Proposition \ref{pbiasedprop}.

\begin{proof}[Proof of Proposition \ref{pbiasedprop}]
\noindent $(i)$ The strict regeneration follows easily from the stick breaking property of RAM models.
By Lemma \ref{lemPPY}, the renewal probabilities $u_n$ are given by
$$u_n =  \mathbb{P}\left(\max_{1 \leq i \leq n} \frac{\varepsilon_i}{P_i} < \frac{\varepsilon}{T_n}\right),$$ 
where $P_i = W_i \prod_{j=1}^{i-1}(1-W_j)$, $T_n = \prod_{j=1}^n (1-W_j)$, and the $\varepsilon_i$'s and $\varepsilon$ are independent standard exponential variables. Note that for each $x>0$,
$$
\mathbb{P}\left(\max_{1 \leq i \leq n} \frac{\varepsilon_i}{P_i} < \frac{x}{T_n}\right)  = \mathbb{P} \left(\bigcap_{i=1}^n \left\{\varepsilon_i < \frac{x P_i}{T_n}\right\} \right)
      = \mathbb{E} \prod_{i=1}^n \left(1 - e^{-\frac{xP_i}{T_n}} \right),
$$
which, by conditioning on $\varepsilon = x$, leads to 
\begin{equation}
\label{interm}
u_n = \int_0^{\infty} e^{-x} \mathbb{E} \prod_{i=1}^n \left(1 - e^{-x P_i / T_n}\right) dx.
\end{equation}
Since
$(W_1, \ldots, W_n) \stackrel{(d)}{=} (W_n, \ldots, W_1)$ for every $n \ge 1$, the formula \eqref{interm} simplifies to \eqref{PTa}.
So to prove $u_{\infty} > 0$, it suffices to prove \eqref{gencri}.

\noindent $(ii)$ This is the deterministic case where $P_i = q^{i-1}(1-q)$ and $T_n = q^n$. So the formula \eqref{PTa} specializes to
$$u_n = \int_0^{\infty} e^{-x} \prod_{i=1}^n \left(1 -  e^{-x(1-q)/q^i}\right) dx.$$
It follows by standard analysis that $u_{\infty}: = \lim_{n \rightarrow \infty} u_n > 0$ if and only if
$$\sum_{i = 1}^{\infty} e^{-x(1-q)/q^i} < \infty.$$
But this is obvious for $0 < q < 1$, which implies that $\Pi$ is positive recurrent.

\noindent $(iii)$ This case corresponds to $P_i = W_i \prod_{j=1}^{i-1}(1-W_j)$ and $T_n = \prod_{j=1}^n (1-W_j)$, where $W_i$ are i.i.d. beta$(1,\theta)$ variables. 
Note that for each $i$, $W_i$ is independent of $T_{i-1}$. By conditioning on $T_{i-1}$, we get
\begin{align*}
\sum_{i = 2}^{\infty} \mathbb{E}  \exp\left(-\frac{x W_i}{T_i} \right)
& = \int_0^1  \mathbb{E}\exp \left(-\frac{xw}{T_{i-1}(1-w)} \right) \cdot \theta(1-w)^{\theta-1} dw  \\
& = \int_0^1 \int_0^1 \exp \left(-\frac{xw}{t (1-w)} \right) \cdot \frac{\theta}{u} \cdot \theta(1-w)^{\theta-1} dt dw,
\end{align*}
where the second equality follows from Ignatov's description \cite{Ignatov} of GEM$(\theta)$ variables as a Poisson point process on $(0,1)$ with intensity $\theta (1-u)^{-1} du$. Note that
$$\int_0^1 \int_0^1  \exp \left(-\frac{xw}{t (1-w)} \right) u^{-1} (1-w)^{\theta-1} dt dw = \int_0^1 E_1\left(\frac{xw}{1-w}\right) (1-w)^{\theta-1}dw,$$
where $E_1(x): = \int_x^{\infty} u^{-1} e^{-u} du$ with $E_1(x) \sim -\log(x)$ as $x \rightarrow 0^+$. It follows by elementary estimates that the above integral is finite, which leads to the desired result.
\end{proof}

Let $\Pi$ be the $P$-biased permutation of $\mathbb{N}_{+}$ for $P$ the geometric$(1-q)$ distribution, with the renewal sequence $(u_{n,q};~n \ge 1)$.
Let $C_{n,1,q}$ be the number of fixed points of $\Pi$ contained in $[n]$, and $\nu_{1,q}$ be the expected number of fixed points of $\Pi$ in a generic component.
According to Proposition \ref{cormain},
$$\lim_{n \rightarrow \infty} \frac{C_{n,1,q}}{n}  = \nu_{1,q} u_{\infty,q}  \quad a.s.,$$
where $u_{\infty,q}: = \lim_{n \rightarrow \infty} u_{n,q}$. Note that with probability $1-q$, a generic component has only one element. This implies that $\nu_{1,q} \ge 1-q$. It follows from Proposition \ref{bon} that $\lim_{q \downarrow 0} u_{\infty,q} = 1$. As a result,
\begin{equation}
\label{fix1}
\lim_{n \rightarrow \infty} \frac{C_{n,1,q}}{n} = \alpha(q) \quad a.s. \quad \mbox{with } \lim_{q \downarrow 0} \alpha(q) = 1.
\end{equation}
Similarly, by letting $\Pi$ be the $P$-biased permutation of $\mathbb{N}_{+}$ for $P$ the GEM$(\theta)$ distribution, and $C_{n,1,\theta}$ be the number of fixed points of $\Pi$ contained in $[n]$,
\begin{equation}
\label{fix2}
 \lim_{n \rightarrow \infty} \frac{C_{n,1,\theta}}{n} = \beta(\theta) \quad a.s. \quad \mbox{with } \lim_{\theta \downarrow 0} \beta(\theta) = 1.
\end{equation}
\section{Regenerative $P$-biased permutations}
\label{s7}

This section provides further analysis of $P$-biased permutations of $\mathbb{N}_{+}$, especially for $P$ the GEM$(1)$ distribution, with the $W_i$'s i.i.d. uniform on $(0,1)$.
While the formulas provided by \eqref{ie}, \eqref{PTa}, or by summing the r.h.s. of \eqref{genbias} over all permutations $\pi \in \mathfrak{S}_k$ for the renewal probabilities $u_k$ and their limit $u_{\infty}$ are quite explicit, it is not easy to evaluate these integrals and their limit directly.
For instance, even in the simplest case where the $W_i$'s are uniform on $(0,1)$, explicit evaluation of $u_k$ for $k \ge 2$ involves the values of $\zeta(j)$ of the Riemann zeta function at $j = 2, \ldots, k$, as indicated later in Proposition \ref{conju}.

We start with an exact simulation of the $P$-biased permutation for any $P = (P_1,P_2, \ldots)$ with $P_i>0$ for all $i\ge 1$ involving the following construction of a process $(\mathcal{W}_k;~ k \ge 1)$ with state space the set of finite unions of open subintervals of $(0,1)$, from $P$ and a collection of i.i.d. uniform variables $U_1, U_2 ,\ldots$ on $(0,1)$ independent of $P$.
\begin{itemize}
\item
Construct $F_j = P_1 + \cdots + P_j$ until the least $j$ such that $F_j > U_1$.
Then set 
$$\mathcal{W}_1 = \left( \bigcup_{i = 1}^{j-1} (F_{i-1},F_i) \right) \cup ( F_j, 1),$$
with convention $F_0: = 0$.
\item 
Assume that $\mathcal{W}_{k-1}$ has been constructed for some $k \ge 2$ as a finite union of open intervals with the rightmost interval $(F_j,1)$ for some $j \ge 1$. 
If $U_k$ lands in one of the intervals of $\mathcal{W}_{k-1}$ that is not the rightmost interval, then remove that interval from $\mathcal{W}_{k-1}$ to create $\mathcal{W}_k$.
If $U_k$ hits the rightmost interval $(F_j,1)$, then construct $F_\ell = F_j + P_{j+1} + \cdots + P_\ell$ for $\ell > j$ until the least $\ell$ such that $F_\ell > U_k$. Set
$$\mathcal{W}_k = (\mathcal{W}_{k-1} \cap (0,F_j)) \cup \left( \bigcup_{i = j+1}^{\ell -1} (F_{i-1}, F_i) \right) \cup (F_{\ell},1).$$
\end{itemize}
It is not hard to see that a $P$-biased permutation $\Pi$ of $\mathbb{N}_{+}$ can be recovered from the process $(\mathcal{W}_k;~ k \ge 1)$ driven by $P$, with $\Pi_k$ a function of $\mathcal{W}_1, \ldots, \mathcal{W}_k$. 
In particular, the length $Y_1$ of the first component of $\Pi$ is 
$$
Y_1 = \min \{k \ge 1: \mathcal{W}_k \mbox{ is composed of a single interval } (F_\ell,1) \mbox{ for some }l\}.
$$
In the sequel, the notation $\stackrel{\theta}{=}$ or $\stackrel{\theta}{\approx}$ indicates exact or approximate evaluations for the GEM$(\theta)$ model; that is the residual factors $W_i$ are i.i.d. beta$(1,\theta)$ distributed.
By a simulation of the process $(\mathcal{W}_k;~ k \ge 1)$ for GEM$(1)$, we get some surprising results:
\begin{equation}
\label{simulation}
\mathbb{E} Y_1 \stackrel{1}{\approx} 3 \quad \mbox{and} \quad Var(Y_1) \stackrel{1}{\approx} 11,
\end{equation}
which suggests that 
\begin{equation}
\label{ukone}
u_{\infty} = 1 / \mathbb{E}Y_1 \stackrel{1}{=} 1/3. 
\end{equation}

These simulation results \eqref{simulation} are explained by the following lemma, which provides an alternative approach to the evaluation of $u_k$ derived from a RAM.
This lemma is suggested by work of Gnedin and coauthors on the Bernoulli sieve \cite{Gnedinsieve,GIM,Gsmall}, and following work on extremes and gaps in sampling from a RAM by Pitman and Yakubovich \cite{PY17, P17}.
\begin{lemma}
Let $X_1, X_2, \ldots$ be a sample from the RAM \eqref{ram} with i.i.d. stick-breaking factors $W_i \stackrel{(d)}{=} W$ for some distribution of $W$ on $(0,1)$.
For positive integers $n$ and $k = 0,1, \ldots$ let 
\begin{equation}
Q_n^*(k):= \sum_{i=1}^n 1 (X_i > k )
\end{equation}
represent the number of the first $n$ balls which land outside the first $k$ boxes. For $m = 1,2, \ldots$ 
let $n(k,m):= \min \{n : Q_n^*(k) = m \}$ be the first time $n$ that there are $m$ balls outside the first $k$ boxes. 
Then:
\begin{itemize}
\item
For each $k$ and $m$ there is the equality of joint distributions
\begin{equation}
\left( Q_{n(k,m)} ^* (k-j), 0 \le j \le k \right) \stackrel{(d)}{=} ( \widehat{Q}_j, 0 \le j \le k \,|\, \widehat{Q}_0 = m )
\end{equation}
where $(\widehat{Q}_0, \widehat{Q}_1, \ldots)$ with $1 \le \widehat{Q}_0 \le \widehat{Q}_1 \cdots$ is a Markov chain with state space $\mathbb{N}_{+}$ and stationary transition probability function
\begin{equation}
\label{hatqdef}
\widehat{q}(m,n) :=  {n - 1\choose m - 1} \mathbb{E} W^{n-m} (1-W)^m \quad \mbox{for}~m \le n.
\end{equation}
So $\widehat{q}(m, \bullet)$ is the mixture of Pascal $(m,1-W)$ distributions, and the distribution of the $\widehat{Q}$ increment from state $m$ is mixed negative binomial $(m, 1-W)$.
\item
For each $k \ge 1$ the renewal probability $u_k$ for the $P$-biased permutation of $\mathbb{N}_{+}$ for $P$ a RAM is the probability that the Markov chain $\widehat{Q}$ started in state $1$
is strictly increasing for its first $k$ steps:
\begin{equation}
\label{ukform}
u_k = \mathbb{P}( \widehat{Q}_0 < \widehat{Q}_1 < \cdots < \widehat{Q}_k \,|\, \widehat{Q}_0 = 1 ).
\end{equation}
\item
The sequence $u_k$ is strictly decreasing, with limit $u_\infty \ge 0$ which is the probability that the Markov chain $\widehat{Q}$ started in state $1$ is strictly increasing forever:
\begin{equation}
\label{uinfform}
u_\infty = \mathbb{P}( \widehat{Q}_0 < \widehat{Q}_1 < \cdots \,|\, \widehat{Q}_0 = 1 ).
\end{equation}
\end{itemize}
\end{lemma}

\begin{proof}
For $0 < v < 1$ and $U_1, U_2, \ldots$ a sequence of i.i.d. uniform $[0,1]$ variables, let
$$
N_n(v,1):= \sum_{i = 1}^n 1 ( v < U_i < 1)
$$
be the number of the first $n$ values that fall in $(v,1)$, and let 
$$g(v,m):=  \min \{n \ge 1 : N_n(v,1) = m \}$$
 be the
random time when $N_n(v,1)$ first reaches $m$. So $g(v,m)$ has the Pascal$(m, 1-v)$ distribution of the sum
of $m$ independent random variables with geometric $(1-v)$ distribution on $\mathbb{N}_{+}$.
Then there is the well known identity in distribution of Pascal counting processes \cite{F79,BR91}
\begin{equation}
\label{pascal}
\left( N_{g(v,m)} (u,1 ) , 0 \le u \le v \right) \stackrel{(d)}{=} \left(    Y_m  \left( \log \left( \frac{ 1- v}{1-u} \right) \right), 0 \le u \le v \right),
\end{equation}
where $(Y_m(t), t \ge 0)$ is a {\em standard Yule process}; that is the pure birth process on positive integers with birth rate $k$ in state $k$, with initial state $Y_m(0) = m$.
Let the sample $X_1, X_2, \ldots$ from the RAM be constructed as $X_i = j$ iff $U_i \in (F_{j-1},F_j]$ where $F_j:= 1 - \prod_{i=1}^j(1-W_i)$ for
a sequence of stick-breaking factors $(W_i;~i \ge 1)$ independent of the uniform sample points $(U_i;~i \ge 1)$. 
Then by construction $Q^{*}_{n(k,m)} (i) = N_{g(F_k,m)} (F_i,1 )$ for each $0 \le i \le k$. 
The identity in distribution 
\eqref{pascal} yields
\begin{equation}
\label{pascalq}
\left( Q^{*}_{n(k,m)} (i) , 0 \le i \le k  \right) \stackrel{(d)}{=} \left(    Y_m  \left( \log \left( \frac{ 1- F_k}{1-F_i} \right) \right), 0 \le i \le k \right),
\end{equation}
first conditionally on $F_1, \ldots, F_k$,  then also unconditionally, where on the right side it is assumed that 
the Yule process $Y_m$ is independent of $F_1, \ldots, F_k$.
By a reversal of indexing, and the equality in distribution $(W_k, \ldots, W_1) \stackrel{(d)}{=} (W_1, \ldots, W_k)$, this gives
\begin{equation}
\label{pascalq}
\left( Q^{*}_{n(k,m)} (k - j) , 0 \le j \le k  \right) \stackrel{(d)}{=} \left(    Y_m  ( \tau_j ), 0 \le j \le k \right),
\end{equation}
where $\tau_j := \sum_{i=1}^j - \log(1-W_i)$, and the $W_i$ are independent of the Yule process $Y_m$. 
It is easily shown that the process on the right side of \eqref{pascalq} is a Markov chain with stationary transition function $\widehat{q}$ as in 
\eqref{hatqdef}. This gives the first part of the lemma, and the remaining parts follow easily.
\end{proof}

For $P$ the GEM$(\theta)$ distribution, the transition probability function $\widehat{q}$ of the $\widehat{Q}$ chain simplifies to
\begin{equation}
\label{gemhatq} 
\widehat{q}(m,n)  \stackrel{\theta}{=}  \frac{ (m)_{n-m} (\theta)_m }{ (1 + \theta)_{n} }  \stackrel{1}{=}  \frac{ m }{ n ( n+1) }  \quad \mbox{for}~m \le n,
\end{equation}
where 
$$(x)_j:= x ( x+1) \cdots (x + j-1) = \frac{\Gamma(x+j)}{\Gamma(x)}. $$
The Markov chain $\widehat{Q}$ with the transition probability function \eqref{gemhatq} for $\theta = 1$ was first encountered
by Erd\"os, R\'enyi and Sz\"usz in their study \cite{ERS} of {\em Engel's series} derived from $U$ with uniform $(0,1)$ distribution, that is
$$
U = \frac{1}{q_1} + \frac{1}{q_1 q_2 } + \cdots + \frac{1 } { q_1 q_2 \cdots q_n} + \cdots ,
$$
for a sequence of random positive integers $q_j \ge 2$. They showed that 
$$
(q_{k+1} - 1, k \ge 0)  \stackrel{(d)}{=}  (\widehat{Q}_k, k \ge 0 ),
$$
for $\widehat{Q}$ with transition matrix $\widehat{q}$ as in \eqref{gemhatq} for $\theta = 1$, and initial distribution 
\begin{equation}
\label{engelinit}
\mathbb{P}(\widehat{Q}_0 = m ) = \frac{ 1}{ m (m+1) } \quad \mbox{for}~ m \ge 1.
\end{equation}
R\'enyi \cite[Theorem 1]{Renyi} showed that for this Markov chain derived from Engel's series, the occupation times 
\begin{equation}
\label{occdef}
G_j:= \sum_{k=0}^\infty 1( \widehat{Q}_k = j) \quad  \mbox{for}~ j \ge 1,
\end{equation}
are independent random variables with geometric$(j/(j+1))$ distributions on $\mathbb{N}_0$.
R\'enyi deduced that with probability one the chain $\widehat{Q}$ is eventually strictly increasing, and \cite[(4.5)]{Renyi} that for the initial distribution \eqref{engelinit} of $\widehat{Q}_0$
\begin{equation}
\label{qone}
\mathbb{P}( \widehat{Q}_0 < \widehat{Q}_1 < \cdots) \stackrel{1}{=} \prod_{j=1}^\infty \mathbb{P}(G_j \le 1) \stackrel{1}{=} \prod_{j=1}^\infty \frac{ j(j+2) }{(j+1)^2} = \frac{1}{2},
\end{equation}
by telescopic cancellation of the infinite product. 
A slight variation of R\'enyi's calculation gives for each possible initial state $m$ of the chain
\begin{equation}
\label{qonem}
\mathbb{P}(\widehat{ Q}_0 < \widehat{Q}_1 < \cdots \,|\, \widehat{Q}_0 = m ) \stackrel{1}{=}  \frac{ m } {m+1 } \prod_{j={m+1}}^\infty \frac{ j (j+2)}{(j+1)^2}  = \frac{m}{m + 2}.
\end{equation}
The instance $m=1$ of this formula, combined with \eqref{uinfform}, proves the formula \eqref{ukone} for $u_\infty$ for the GEM$(1)$ model.
A straightforward variation of these calculations gives the corresponding result for $P$ the GEM$(\theta)$ distribution:
\begin{equation}
\label{uktheta}
u_\infty \stackrel{\theta}{=}  \frac{1}{1+\theta} \prod_{j =2}^{\infty} \frac{j(j+2 \theta)}{(j+\theta)^2}= \frac{\Gamma(\theta+2)\Gamma(\theta+1)}{\Gamma(2 \theta+2)}.
\end{equation}
A key ingredient in this evaluation is the fact that in the GEM$(\theta)$ model the random occupation times $G_j$ of $\widehat{Q}$ are independent geometric variables, see \cite{PY17,P17}.
For a more general RAM, the $G_j$'s may not be independent, and they may not be exactly geometric, only conditionally so given $G_j \ge 1$.
The Yule representation \eqref{pascalq} of $\widehat{Q}$ given $\widehat{Q}(0) = 1$ as $\widehat{Q}(j) = Y_1(\tau_j)$ combined with Kendall's representation \cite[Theorem 1]{Kendall66} of $Y_1(t) = 1+ N( \varepsilon  (e^t - 1) )$ for $N$ a rate 1 Poisson process and $\varepsilon$ standard exponential independent of $N$, only reduces the expression \eqref{uinfform} for $u_\infty$ back to the limit form as $n \to \infty$ of the previous expression \eqref{PTa}.  
So $u_\infty$ for a RAM is always an integral over $x$ of the expected value of an infinite product of random variables.
See also \cite{Ik1,Ik2,PY17} for treatment of closely related problems.
Now we give a proof of Proposition \ref{simplecond}.
\begin{proof}[Proof of Proposition \ref{simplecond}]
The result of \cite[Theorem 3.3]{Gsmall} shows that under the assumptions of the proposition,
if $L_n$ is the number of empty boxes to the left of the rightmost box when $n$ balls are thrown, then 
$$
L_n \stackrel{(d)}{\longrightarrow} L_{\infty}:= \sum_{j=1}^\infty (G_j - 1)_{+}  \quad \mbox{ as } n \to \infty,
$$
where the right side is defined by the occupation counts \eqref{occdef} of the Markov chain $\widehat{Q}$ for the special entrance law 
\begin{equation}
\label{entrancelaw}
\mathbb{P}( \widehat{Q}_0 = m ) =  \frac{ \mathbb{E} W^m }{ m \, \mathbb{E} \left[ - \log(1 - W) \right]} \quad \mbox{for} ~ m \ge 1,
\end{equation}
which is the limit distribution of $Z_n$, the number of balls in the rightmost occupied box, as $n \to \infty$, and that also
$$
\mathbb{E} L_n \to \mathbb{E} L_\infty = \frac{ \mathbb{E} \left[ - \log W \right] }{ \mathbb{E} \left[- \log (1-W)\right] },
$$
which is finite by assumption. It follows that $\mathbb{P}(L_\infty < \infty ) = 1$,
hence also that $\mathbb{P}_m(L_\infty < \infty ) = 1$ for every $m$,
where $\mathbb{P}_m(\bullet):= \mathbb{P}(\bullet \,|\, \widehat{Q}_0 = m)$.
Let $R:= \max \{j : G_j > 1 \}$ be the index of the last repeated value of the Markov chain. 
From $\mathbb{P}_1(L_\infty < \infty) = 1$ it follows that $\mathbb{P}_1(R < \infty) = 1$,
hence that $\mathbb{P}_1 (R  = r ) > 0$ for some positive integer $r$.
But  for $r = 2, 3, \ldots$, a last exit decomposition gives
$$
\mathbb{P}_1(R = r)   = \left( \sum_{k=1}^\infty \mathbb{P}_1( \widehat{Q}_{k-1} = \widehat{Q}_k = r )  \right)  \mathbb{P}_r ( L_\infty  = 0 ),
$$
where both factors on the right side must be strictly positive to make $\mathbb{P}_1(R=r) >0$. Combined
with a similar argument if $\mathbb{P}_1(R = 1) >0$, this implies $\mathbb{P}_r ( L_\infty  = 0 ) >0$ for some $r \ge 1$,
hence also
$$ 
u_\infty = \mathbb{P}_1( L_\infty = 0 )  \ge \mathbb{P}_1( G_i \le 1 \mbox{ for } 0 \le i <r, G_r = 1) \mathbb{P}_r (L_\infty = 0) >0,
$$
which is the desired conclusion.
\end{proof}

To conclude, we present explicit formulas for $u_k$ of a GEM$(1)$-biased permutation of $\mathbb{N}_{+}$. The proof is deferred to the forthcoming article \cite{DPT}.
\begin{proposition} \cite{DPT}
\label{conju}
Let $\Pi$ be a GEM$(1)$-biased permutation of $\mathbb{N}_{+}$, with the renewal sequence $(u_k;~ k \geq 0)$. Then $(u_k;~k \ge 0)$ is characterized by any one of the following equivalent conditions:
\begin{enumerate}[$(i).$]
\item
The sequence $(u_k;~ k \ge 0)$ is defined recursively by
\begin{equation}
\label{rec}
2 u_{k} + 3 u_{k-1} + u_{k-2} \stackrel{1}{=} 2 \zeta(k) \quad \mbox{with } u_0 = 1, \, u_1 = 1/2,
\end{equation}
where $\zeta(k): = \sum_{n = 1}^{\infty} 1/n^k$ is the Riemann zeta function.
\item
For all $k \ge 0$,
\begin{equation}
\label{rzs}
u_{k} \stackrel{1}{=} (-1)^{k-1} \left(2 - \frac{3}{2^k} \right) + \sum_{j=2}^{k}  (-1)^{k-j} \left(2 -\frac{1}{2^{k-j}} \right) \zeta(j).
\end{equation}
\item
For all $k \ge 0$,
\begin{equation}
\label{positive}
u_{k} \stackrel{1}{=} \sum_{j = 1}^{\infty} \frac{2}{j^k(j+1)(j+2)}.
\end{equation}
\item
The generating function of $(u_{k};~ k \ge 0)$ is 
\begin{equation}
\label{Uz}
U(z) : = \sum_{k=0}^{\infty}u_k z^k \stackrel{1}{=} \frac{2}{(1+z)(2+z)} \Bigg[ 1 +   \Bigg(2 - \gamma - \Psi(1-z)\Bigg) z\Bigg],
\end{equation}
where 
$\gamma: = \lim_{n \rightarrow \infty} (\sum_{k=1}^n 1/k - \ln n) \approx 0.577$ is the {\em Euler constant}, and
$\Psi(z): = \Gamma'(z)/\Gamma(z)$ with
$\Gamma(z): = \int_0^\infty t^{z-1} e^{-t} dt$, is the {\em digamma function}. 
\end{enumerate}
\end{proposition}

The distribution of $Y_1$, that is $f_{k}: = \mathbb{P}(Y_1 = k)$ for all $k \ge 1$, is determined by $(u_k;~ k \ge 0)$ or $U(z)$ via the relations \eqref{ufrecursion}-\eqref{UFrelation}.
It is easy to see that the generating function $F(z)$ of $(f_k;~ k \ge 1)$ is real analytic on $(0,z_0)$ with $z_0 \approx 1.29$.
This implies that all moments of $Y_1$ are finite.
By expanding $F(z)$ into power series at $z=1$, we get:
\begin{equation}
\label{powerseries}
F(z)  \stackrel{1}{=}  1 + 3(z-1) + \frac{17}{2} (z-1)^2 + \frac{1}{2}(47 + \pi^2) (z-1)^3 + \cdots,
\end{equation}
which agrees with the simulation \eqref{simulation}, since $\mathbb{E}Y_1 = F'(1) \stackrel{1}{=} 3$ and $Var(Y_1) = 2 F''(1) + F'(1) - F'(1)^2 \stackrel{1}{=}11$.

\bibliographystyle{plain}
\bibliography{Regenerative}
\end{document}